\documentclass[sn-mathphys]{sn-jnl}

\jyear{2022}%

\usepackage{amsmath}
\usepackage{graphicx,wrapfig}
\usepackage{amssymb,amsgen,amsfonts}
\usepackage{multirow}
\usepackage{epsfig,epstopdf,color,bm}

\theoremstyle{thmstyleone}
\newtheorem{theorem}{Theorem}
\newtheorem{lemma}[theorem]{Lemma}
\theoremstyle{thmstyletwo}%
\newtheorem{remark}{Remark}%

\theoremstyle{thmstylethree}%

\newcommand{\eps}{\varepsilon}
\newcommand{\bx}{{\bf x}}
\newcommand{\nn}{\nonumber}

\numberwithin{equation}{section}
\numberwithin{theorem}{section}
\numberwithin{remark}{section}
\numberwithin{table}{section}
\numberwithin{figure}{section}

\graphicspath{{figures/}}

\begin{document}

\title[Improved uniform error bounds on LEI-FP for SGE]{Improved uniform error bounds on a Lawson-type exponential integrator for the long-time dynamics of sine--Gordon equation}

\author[1]{\fnm{Yue} \sur{Feng}}\email{yue.feng@sorbonne-universite.fr}

\author[1]{\fnm{Katharina} \sur{Schratz}}\email{katharina.schratz@sorbonne-universite.fr}

\affil[1]{\orgdiv{Laboratoire Jacques-Louis Lions}, \orgname{Sorbonne Universit\'e}, \orgaddress{\street{4 place Jussieu}, \city{Paris}, \postcode{75007},  \country{France}}}


\abstract{We establish the improved uniform error bounds on a Lawson-type exponential integrator Fourier pseudospectral (LEI-FP) method for the long-time dynamics of sine--Gordon equation where the amplitude of the initial data is $O(\eps)$ with $0 < \eps \ll 1$ a dimensionless parameter up to the time at $O(1/\eps^2)$. The numerical scheme combines a Lawson-type exponential integrator in time with a Fourier pseudospectral method for spatial discretization, which is fully explicit and efficient in practical computation thanks to the fast Fourier transform. By separating the linear part from the sine function and employing the regularity compensation oscillation (RCO) technique which is introduced to deal with the polynomial nonlinearity by phase cancellation, we carry out the improved error bounds for the semi-discreization at $O(\eps^2\tau)$ instead of $O(\tau)$ according to classical error estimates and at $O(h^m+\eps^2\tau)$ for the full-discretization up to the time $T_{\eps} = T/\eps^2$ with $T>0$ fixed. This is the first work to establish the improved uniform error bound for the long-time dynamics of the NKGE with non-polynomial nonlinearity. The improved error bound is extended to an  oscillatory sine--Gordon equation with $O(\eps^2)$ wavelength in time and $O(\eps^{-2})$ wave speed, which indicates that the temporal error is independent of $\eps$ when the time step size is chosen as $O(\eps^2)$. Finally, numerical examples are shown to confirm the improved error bounds and to demonstrate that they are sharp.}




\maketitle

\section{Introduction}
\label{sec:int}
In this paper, we consider the following sine--Gordon equation (SGE)
\begin{equation}
\label{eq:SIE}
\begin{cases}
\partial_{tt}u(\bx, t)-\Delta u({\bx}, t)+ \sin(u({\bx}, t))=0,\quad \bx \in \Omega,\quad t > 0,\\
u(\bx, 0) = \eps \phi(\bx) = O(\eps), \quad \partial_t u(\bx, 0) = \eps \gamma(\bx)=O(\eps),\quad {\bx} \in \Omega,
\end{cases}
\end{equation}
where $t$ is time, $\bx$ is the spatial coordinate, $\Delta$ is the Laplace operator, $u:= u(\bx, t)$ is a real-valued scalar field, $\varepsilon\in (0, 1]$ is a dimensionless parameter used to characterize the amplitude of the initial data and $\Omega =  \prod_{i = 1}^{d} (a_i, b_i) \subset \mathbb{R}^d$ $(d = 1, 2, 3)$ is a bounded domain with periodic boundary condition. In the initial data, $\phi(\bx)$ and $\gamma(\bx)$ are two given real-valued functions which are independent of $\eps$. 

The sine--Gordon equation is a special case of the nonlinear Klein--Gordon equation (NKGE) and used to describe ubiquitous phenomena in many fields. It arises in the propagation of fluxion in Josephson junctions between two superconductors, dislocations in crystals, vibrations of DNA molecules, laser pulses, and quantum field theory, etc \cite{CLL,Josephson,SCR,XJX,YLV}. 
A remarkable property of the sine--Gordon \eqref{eq:SIE} is the conservation of the energy as
\begin{align}
E(t) &:= E(u(\cdot, t)) = \int_{\Omega}\Big[\vert\partial_t u({\bx}, t)\vert^2 + \vert\nabla u({\bx}, t)\vert^2 + 2(1-\cos(u({\bx}, t)))\Big] d{\bx} \nn \\
& \equiv  \int_{\Omega}\Big[\eps^2\vert\gamma({\bx})\vert^2 + \eps^2\vert\nabla \phi({\bx})\vert^2 + 2(1-\cos(\eps \phi({\bx})))\Big] d{\bx} \nn \\
& = E(0)=O(\eps^2), \quad t \geq 0.
\end{align}

In the past decades, a surge of analytical and numerical results for the sine--Gordon equation have been reported in the literature. In the analytical aspect, the soliton solutions of the sine--Gordon equation were well studied \cite{ADM,Deh,Hirota,Lei,VNK,Wazwaz,ZJ}. Along the numerical front, various numerical schemes were proposed and analyzed including the finite difference method \cite{GPRV,FV,MD}, finite element method \cite{AHH,Tou}, spectral method \cite{BD,CG} and Adomian's decomposition method \cite{DK,Kaya}, etc. For more details, we refer to \cite{AHS,BD,Duncan,KL,LV,SKV} and references therein.

In order to study the long-time asymptotics of the sine--Gordon equation with small norm solutions \cite{CLL}, one can use the Taylor expansion $\sin(u) = u - \frac{u^3}{6}+O(u^5)$. The leading order behavior of the solution is given by 
\begin{equation}
\partial_{tt} u(\bx, t) - \Delta u({\bx}, t) +  u({\bx}, t) - \frac{ u^3({\bx}, t)}{6} = 0,
\end{equation}
which implies the lifespan of the sine--Gordon equation \eqref{eq:SIE} is at least up to $O(1/\eps^2)$ based on the results for the NKGE with cubic nonlinearity \cite{DS,FZ}.

When $0 < \eps \ll 1$, by introducing $w(\bx, t)=u(\bx, t)/\eps$, the sine--Gordon equation  \eqref{eq:SIE} with $O(\eps)$ initial data and $O(1)$ nonlinearity can be reformulated into the following SGE
\begin{equation}\label{eq:WNE}
\begin{cases}
\partial_{tt} w({\bx}, t)-\Delta w({\bx}, t)+ \frac{1}{\eps}\sin(\eps w({\bx}, t)) = 0, \quad \bx \in \Omega,\quad t > 0, \\
w({\bx}, 0) = \phi({\bx}) = O(1),\quad \partial_t w({\bx}, 0) = \gamma({\bx}) = O(1),\quad {\bx} \in \Omega.
\end{cases}
\end{equation}
In fact, the long-time dynamics of the sine--Gordon equation \eqref {eq:WNE} is equivalent to that of the sine--Gordon equation \eqref{eq:SIE}.

In addition, introducing a re-scale in time
\begin{equation}
\label{rstime}
t=\frac{s}{\varepsilon^2}\Leftrightarrow s=\varepsilon^2 t, \qquad v(\bx,s) = w(\bx,t),
\end{equation}
we can re-formulate the sine--Gordon equation \eqref{eq:WNE} into the following oscillatory sine--Gordon equation
\begin{equation}
\left\{
\begin{split}
&\eps^4\partial_{ss} v(\bx, s) -\Delta v(\bx, s) + \frac{1}{\eps}\sin(\eps v(\bx, s))=0,\quad \bx \in \Omega,\quad s > 0,\\
&v(\bx, 0) = \phi(\bx)=O(1),\quad \partial_s v(\bx, 0) = \frac{1}{\eps^2}\gamma(\bx)=O({\eps}^{-2}),\quad \bx \in \Omega.
\end{split}\right.
\label{eq:HOE}
\end{equation}
\begin{figure}[ht!]
\begin{minipage}{0.49\textwidth}
\centerline{\includegraphics[width=6.5cm,height=5.5cm]{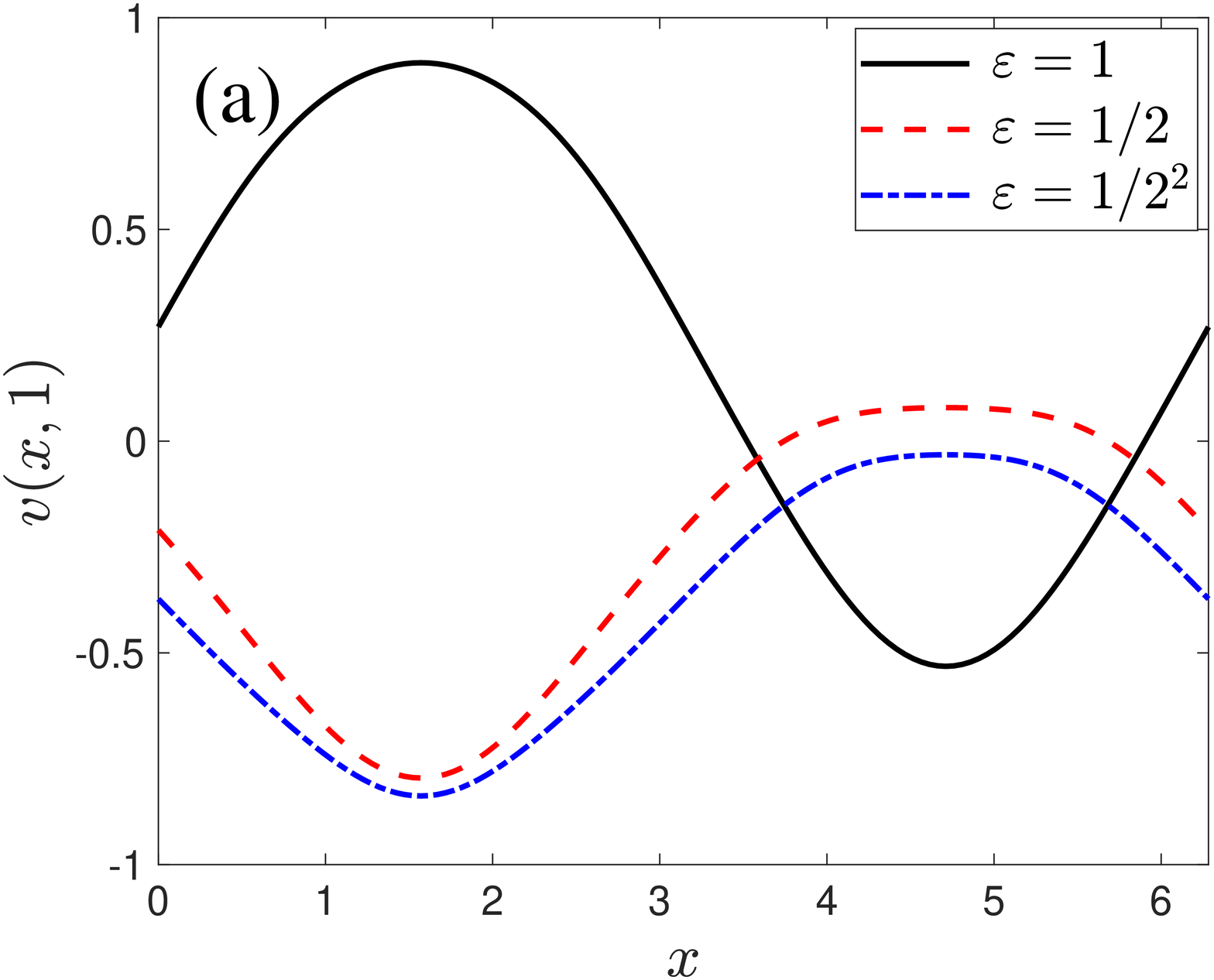}}
\end{minipage}
\begin{minipage}{0.49\textwidth}
\centerline{\includegraphics[width=6.5cm,height=5.5cm]{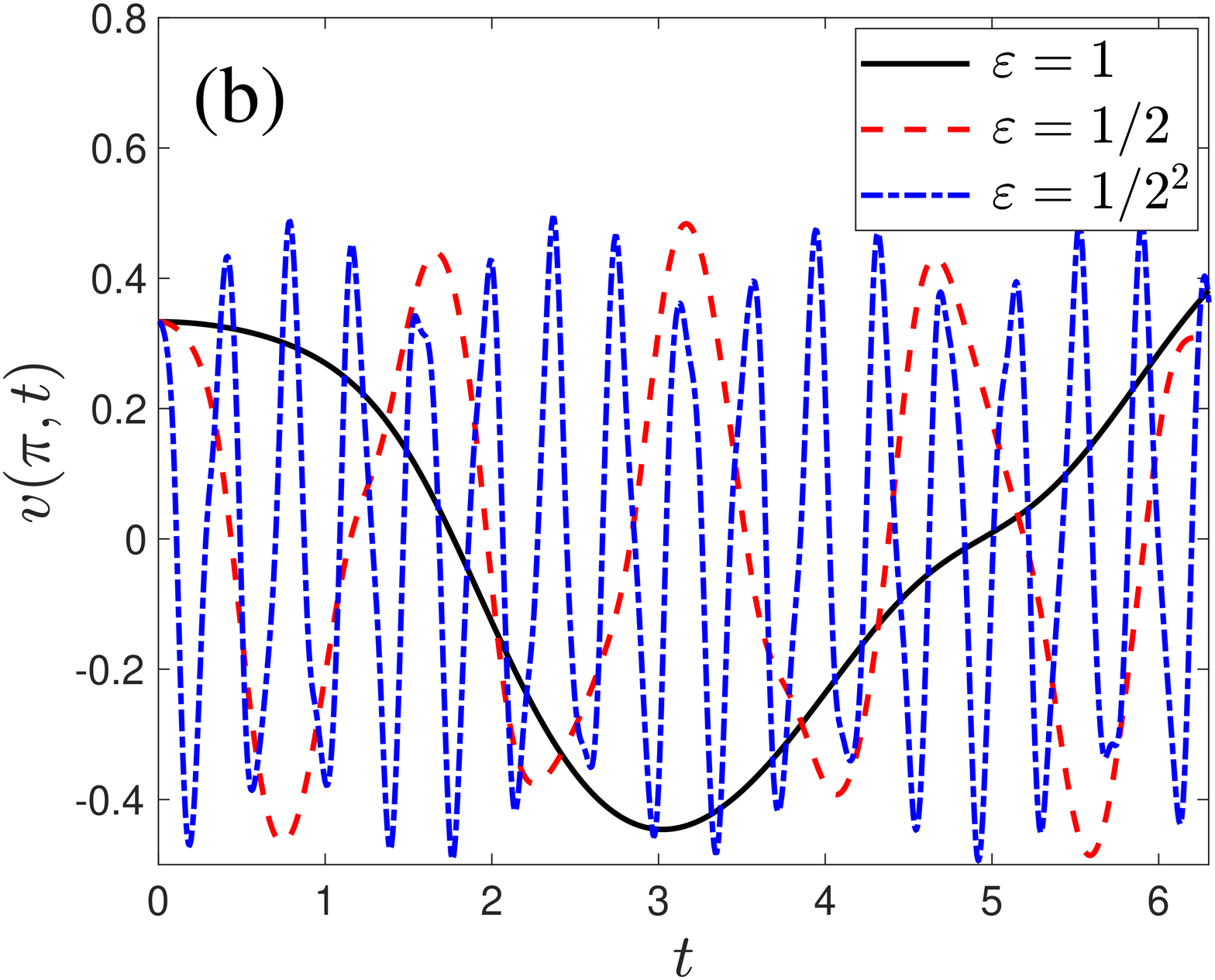}}
\end{minipage}
\caption{The solution $v(x, 1)$ and $v(\pi, s)$ of the oscillatory SGE \eqref{eq:HOE} in 1D with different $\eps$.}
\label{fig:1D_HOE}
\end{figure}
The solution of the oscillatory SGE \eqref{eq:HOE} propagates waves with amplitude at $O(1)$, wavelength at $O(1)$ and $O(\eps^2)$ in space and time, respectively, and wave speed at $O(\eps^{-2})$. Fig. \ref{fig:1D_HOE} shows the oscillatory nature of the SGE \eqref{eq:HOE}.

In our recent work, long-time error bounds for the NKGE with weak nonlinearity have been rigorously established for finite difference time domain (FDTD) methods \cite{BFY}, exponential wave integrator Fourier pseudospectral (EWI-FP) method \cite{FY}, and time-splitting Fourier pseudospectral (TSFP)  method \cite{BCF,BaoFS}. For the time-splitting methods applied to the NKGE with power-type nonlinearity, the improved uniform error bounds were carried out by introducing the regularity compensation oscillation (RCO) technique which controls high frequency modes by the regularity of the exact solution and low frequency modes by phase cancellation and energy method. Later, the RCO technique is extended to establish the improved error bounds for the long-time dynamics of the time-splitting methods for the (nonlinear) Schr\"odinger equation and (nonlinear) Dirac equation \cite{BCF2, BFYIN}. However, the new analysis technique takes the advantage of the polynomial nonlinearity and it can not be directly extended to deal with the non-polynomial nonlinearity. More specifically, we no longer just summate the local truncation error in each time step, which can establish the uniform error bound in the long-time regime. Instead, we consider the global error and exploit the structure of the polynomial nonlinearity to get the phase cancellation, which leads to the improved uniform error bound. As far as we know, improved uniform error bounds have not been proven on the exponential integrator for the long-time dynamics for the NKGE, especially that with non-polynomial nonlinearity. 

The aim of this paper is to rigorously carry out the improved uniform error bounds on the Lawson-type exponential integrator for the long-time dynamics of the SGE \eqref{eq:SIE} or \eqref{eq:WNE} with the aid of the RCO technique up to the time at $T_{\eps} = T/\eps^2$. It is important to mention that we need to separate a linear part from the sine function of the SGE \eqref{eq:WNE} in the numerical scheme, otherwise we can not obtain the improved uniform error bounds for the long-time dynamics. 

The rest of this paper is organized as follows. In section 2, we separate a linear part from the sine function and reformulate it into a relativistic nonlinear Schr\"odinger equation (NLSE) and then apply a Lawson-type exponential integrator to discretize the SGE in time followed by the full-discretization with the Fourier pseudospectral method for spacial discretization. In section 3, we exploit the regularity compensation oscillation (RCO) technique to establish the improved uniform error bounds for the semi-discretization and full-discretization up to the time at $O(1/\eps^2)$. In section 4, numerical results for the long-time dynamics of the SGE \eqref{eq:WNE} and the dynamics of the oscillatory SGE \eqref{eq:HOE} are presented to confirm the error estimates. Finally, some conclusions are drawn in section 5. Throughout this paper, the notation $A \lesssim  B$ is used to represent that there exists a generic constant $C > 0$ independent of the mesh size $h$, time step $\tau$, $\eps$, and $\tau_0$ such that $\vert A\vert \leq CB$.

\section{A Lawson-type exponential integrator}
In this section, we first separate a linear part from the sine function of the SGE \eqref{eq:WNE} and reformulate it into a relativistic NLSE. For the relativistic NLSE, we discretize it in time by a Lawson-type exponential integrator followed by a full-discretization with the Fourier pseudospectral method in space. Since the long-time dynamics of the SGE \eqref {eq:SIE} is equivalent to that of the SGE \eqref{eq:WNE}, we only show the numerical method and corresponding analysis for the SGE \eqref{eq:WNE} in one dimension (1D) for simplicity of presentation. Generalization to higher dimensions and/or the SGE \eqref{eq:SIE} is straightforward.  In 1D, the SGE \eqref{eq:WNE} on the computational domain $\Omega = (a, b)$ is given as
\begin{equation}
\label{eq:21}
\left\{
\begin{aligned}
&\partial_{tt} w(x, t) - \partial_{xx} w(x, t) +  \frac{1}{\eps}\sin(\eps w(x, t)) = 0,\ x \in \Omega,\ t > 0, \\
&w(a, t) = w(b, t), \quad \partial_x w(a, t) = \partial_x w(b, t), \qquad
t\ge0,\\
&w(x, 0) = \phi(x), \quad \partial_t w(x, 0) = \gamma(x) , \quad x \in \overline{\Omega} = [a, b].
\end{aligned}\right.
\end{equation}

For an integer $m\ge 0$, we denote by $H^m(\Omega)$ the space of functions $u(x)\in L^2(\Omega)$ with finite $H^m$-norm $\|\cdot\|_{H^m}$ given by
\begin{equation}
\label{sn}
\|u\|_{H^m}^2=\sum\limits_{l \in \mathbb{Z}} (1+\mu_l^2)^m\vert\widehat{u}_l\vert^2,\quad \mathrm{for}\ u(x)=\sum\limits_{l\in \mathbb{Z}} \widehat{u}_l e^{i\mu_l(x-a)},\ \mu_l=\frac{2\pi l}{b - a},
\end{equation}
where $\widehat{u}_l \ (l\in \mathbb{Z})$ are the Fourier  coefficients  of the function $u(x)$ \cite{BCZ,BaoFS}. In fact, the  space $H^m(\Omega)$ is the subspace of  classical Sobolev space $W^{m,2}(\Omega)$, which consists of functions with derivatives of order up to $m-1$ being $(b - a)$-periodic. Since we consider the SGE with periodic boundary condition, the above space $H^m(\Omega)$ is suitable. 

Define the operator
\begin{equation}
\langle \nabla \rangle=\sqrt{1-\Delta},
\end{equation}
through its action in the Fourier space by \cite{BaoFS,BFS,FS}:
\begin{equation*}
\langle \nabla \rangle u(x)=\sum\limits_{l\in\mathbb{Z}}\sqrt{1+\mu_l^2}\widehat{u}_l e^{i\mu_l(x-a)}, \quad \mathrm{for}\quad u(x)=\sum\limits_{l\in\mathbb{Z}} \widehat{u}_l e^{i\mu_l(x-a)},\quad x\in[a,b],
\end{equation*}
and the operator $ \langle \nabla \rangle^{-1}$  as
\begin{equation*}
\langle \nabla \rangle^{-1} u(x)=\sum\limits_{l\in\mathbb{Z}}\frac{\widehat{u}_l}{\sqrt{1+\mu_l^2}} e^{i\mu_l(x-a)},\qquad x\in [a, b].
\end{equation*}
With this notation, by separating a linear part $w$ from $\frac{1}{\eps}\sin(\eps w(x, t))$ in the SGE \eqref{eq:21}, we can write it as
\begin{equation}
\partial_{tt} u(x, t) + \langle\nabla \rangle^2 u(x, t) + \frac{1}{\eps}\sin(\eps w(x, t)) - w(x, t) = 0,\quad x \in\Omega,\quad t>0.
\label{eq:R1}
\end{equation}

Denoting $z(x, t) = \partial_t u(x, t)$ and 
\begin{equation}
\psi(x, t) = u(x, t) - i\langle\nabla\rangle^{-1}z(x, t), \quad x\in[a,b], \quad t \ge 0,
\label{eq:psi}
\end{equation}
the SGE \eqref{eq:R1} can be reformulated into the following relativistic NLSE
\begin{equation}
\label{eq:NLS}
\left\{
\begin{aligned}
&i\partial_t \psi(x, t) + \langle\nabla \rangle \psi(x, t) + \langle\nabla \rangle^{-1} f\Big(\frac{1}{2} \left(\psi + \overline{\psi}\right)\Big)(x, t) = 0, \\
&\psi(a,t)=\psi(b,t), \quad \partial_x \psi(a,t)=\partial_x \psi(b,t), \quad t\ge0,\\
&\psi(x,0)=\psi_0(x):=\phi(x)-i\langle\nabla \rangle^{-1}\gamma(x), \quad x\in[a,b],
\end{aligned}\right.
\end{equation}
where $f(\phi)=\frac{1}{\eps}\sin(\eps\phi)-\phi$ and $\overline{\psi}$ denotes the complex conjugate of $\psi$. By the definition \eqref{eq:psi}, the solution of the SGE \eqref{eq:21} can be recovered by
\begin{equation}
\label{eq:wz}
w(x,t)=\frac{1}{2}\left(\psi(x,t)+\overline{\psi}(x,t)\right), \qquad
z(x, t)=\frac{i}{2}\langle\nabla \rangle\left(\psi(x,t)-\overline{\psi}(x,t)\right).
\end{equation}

\subsection{Semi-discretizaiton by a Lawson-type exponential integrator}
For  the relativistic NLSE \eqref{eq:NLS}, we utilize a Lawson-type exponential integrator (LEI) to discretize it in time. In the rest of this paper, we take $\psi(t) = \psi(x, t)$ for notational simplicity, i.e., omit the spatial variable when there is no confusion.

Let $\tau > 0$ be the time step size and take $t_n = n\tau$ for $n = 0, 1, \ldots$. By Duhamel's formula, the exact solution of the relativistic NLSE \eqref{eq:NLS} is given as 
\begin{equation}
\psi(t_n+\tau) = e^{i\tau\langle\nabla\rangle}\psi(t_n) + \int^{\tau}_0 e^{i(\tau-\sigma)\langle\nabla\rangle} F(\psi(t_n + \sigma)) d \sigma,
\label{eq:exact}
\end{equation} 
where the function $F$ is defined by
\begin{equation}
F(\phi) = i\langle\nabla\rangle^{-1}g(\phi), \quad g(\phi) = f\left(\frac{1}{2}(\phi+\overline{\phi})\right).
\label{eq:Fg}
\end{equation}
Denote by $\psi^{[n]}:= \psi^{[n]}(x)$ the approximation of $\psi(x, t_n)$ for $n = 0, 1, \ldots$. Applying the approximation $\psi(t_n + \sigma) \approx \psi(t_n)$ and the first-order Lawson method \cite{Lawson, OS}, we get the first-order Lawson-type exponential integrator (LEI) scheme as
\begin{equation}
\psi^{[n+1]} = \mathcal{L}_{\tau}(\psi^{[n]}):= e^{i\tau\langle\nabla\rangle}\psi^{[n]} + \tau e^{i\tau\langle\nabla\rangle} F(\psi^{[n]}),
\label{eq:semi1}
\end{equation} 
with $\psi^{[0]}=\psi_0=\phi-i\langle\nabla\rangle^{-1}\gamma$. Then, the first-order semi-discretization of the SGE \eqref{eq:21} is
\begin{equation}
\label{eq:semi2}
w^{[n]}=\frac{1}{2}\left(\psi^{[n]}+\overline{\psi^{[n]}}\right), \quad
z^{[n]} = \frac{i}{2}\langle\nabla \rangle\left(\psi^{[n]}-\overline{\psi^{[n]}}\right),\quad n=0,  1, \ldots,
\end{equation}
where $w^{[n]}:=w^{[n]}(x)$ and $z^{[n]}:= v^{[n]}(x)$ are the approximations of $w(x, t_n)$ and $\partial_t w(x, t_n)$ for $n=0, 1, \ldots$, respectively.

\subsection{Full-discretization by Fourier pseudospectral method}
Denote the index set $\mathcal{T}_M = \{l~\vert~l = -M/2,-M/2+1, \cdots, M/2-1\}$, and define $C_p(\Omega) = \{u \in C(\Omega) \vert u(a) = u(b)\}$ and
\begin{equation*}
\begin{split}
& X_M := \mbox{span}\{e^{i\mu_l(x-a)},\ \mu_l = \frac{2 \pi l}{b-a}, x \in \overline{\Omega}, l \in \mathcal{T}_M \}, \\
&Y_M := \{u = (u_0, u_1, \cdots, u_M) ~\vert~ u_0 = u_M\}\subseteq  \mathbb{R}^{M+1}.
\end{split}
\end{equation*}
  For any $u(x) \in C_p(\Omega)$ and a vector $u \in Y_M$, let $P_M : L^2(\Omega) \to X_M$ be the standard $L^2$-projection operator onto $X_M$ and $I_M : C_p(\Omega) \to X_M$ or $I_M : Y_M \to X_M$ be the trigonometric interpolation operator \cite{STL}, i.e.,
\begin{equation*}
(P_M u)(x) = \sum_{l \in \mathcal{T}_M} \widehat{u}_l e^{i\mu_l (x-a)}, \quad (I_M u)(x) = \sum_{l \in \mathcal{T}_M} \widetilde{u}_l e^{i\mu_l (x-a)}, 
\end{equation*}
where $\widehat{u}_l $ and $\widetilde{u}_l $ are the Fourier and discrete Fourier transform coefficients, respectively, defined as
\begin{equation*}
\widehat{u}_l = \frac{1}{b-a} \int^{b}_{a} u(x) e^{-i\mu_l (x-a)} dx, \quad \widetilde{u}_l = \frac{1}{M} \sum^{M-1}_{j=0} u_j e^{-i\mu_l (x_j-a)},\quad l \in \mathcal{T}_M,
\end{equation*}
with $u_j$ interpreted as $u(x_j)$ when involved.
Choose the spatial mesh size $h := \Delta x = (b-a)/M$ with $M$ an even positive integer, and denote the grid points as 
\begin{equation*}
x_j := a+j h,\quad j \in \mathcal{T}^0_M = \{ j~\vert~j = 0, 1, \ldots, M\}. 	
\end{equation*}

Let $\psi_j^n$ be the numerical approximation of $\psi(x_j, t_n)$ for $j\in \mathcal{T}^0_M$ and $n \ge 0$, and denote $\psi^n=(\psi_0^n, \psi_1^n,\ldots, \psi_M^n)^T\in \mathbb{C}^{M+1}$ for $n=0,1,\ldots$. Then, the full-discretization for the relativistic NLSE \eqref{eq:NLS} via the Lawson-type exponential integrator for temporal discretization combining with the Fourier pseudospectral method for spacial discretization is given as
\begin{equation}
\psi^{n+1}_j = \sum_{l \in \mathcal{T}_M}\widetilde{\psi}^{n+1}_l e^{i\mu_l(x_j-a)},	\quad j \in \mathcal{T}^0_M, \quad n=0,1,\ldots,
\label{eq:psifull1}
\end{equation}
where
\begin{equation}
\label{eq:psifull2}
\widetilde{\psi}^{n+1}_l= e^{i\tau\delta_l}\widetilde{\psi}^{n}_l+ \tau e^{i\tau\delta_l}(\widetilde{F(\psi^{n})})_l, \quad (\widetilde{F(\psi^{n})})_l = \frac{i}{\delta_l} (\widetilde{g(\psi^{n})})_l, 
\end{equation}
with $\delta_l=\sqrt{1+\mu_l^2}$ for $l\in \mathcal{T}_M$ and
\[\psi_j^0=\phi(x_j)-i\sum_{l \in \mathcal{T}_M}\frac{\widetilde{\gamma}_l}{\sqrt{1+\mu_l^2}} e^{i\mu_l(x_j-a)}, \qquad j \in \mathcal{T}^0_M. \]

Let $w^n_j$ and $z^n_j$ for $j \in \mathcal{T}^0_M$ and $n \geq 0$ be the approximations of $w(x_j, t_n)$ and  $\partial_t w(x_j, t_n)$, respectively, and denote $w^n = (w^M_0, w^n_1, \ldots, w^n_M)^T \in \mathbb{R}^{M+1}$ and $z^n = (z^M_0, z^n_1, \ldots, z^n_M)^T \in \mathbb{R}^{M+1}$. For $j \in \mathcal{T}^0_M$, taking $w^0_j = \phi(x_j)$ and $z^0_j = \gamma(x_j)$, combining \eqref{eq:psi} and \eqref{eq:psifull1}--\eqref{eq:psifull2}, we obtain the full-discretization of the SGE \eqref{eq:21} by the Lawson-type exponential integrator Fourier pseudospectral (LEI-FP) method as
\begin{equation}
\label{eq:wfull}
\begin{split}
&w_j^{n+1}=\frac{1}{2}\left(\psi_j^{n+1}+\overline{\psi_j^{n+1}}\right),\\
&z_j^{n+1} = \frac{i}{2}\sum_{l \in \mathcal{T}_M}\delta_l\big[\widetilde{(\psi^{n+1})}_l-
\widetilde{(\overline{\psi^{n+1}})}_l\big]\;
e^{i\mu_l (x_j-a)},
\end{split}
\qquad j \in \mathcal{T}^0_M, \quad n\ge 0.
\end{equation}

The LEI-FP scheme is explicit and very efficient thanks to the fast discrete Fourier transform. The memory cost is $O(M)$ and the computational cost per time step is $O(M\log M)$.

\section{Improved uniform error estimates}
We make the following assumption on the exact solution $w:=w(x, t)$ of the SGE \eqref{eq:21} up to the time at $T_{\eps} = T/\eps^2$ with $T>0$ fixed:
\begin{equation*}
{\rm(A)}\qquad
\|w\|_{L^{\infty}\left([0, T_{\eps}]; H^{m+1}\right)} \lesssim 1,\quad   \|\partial_t w\|_{L^{\infty}\left([0, T_{\eps}]; H^{m}\right)} \lesssim 1, \quad m \geq 0,
\end{equation*}
then we will establish the improved uniform error bounds for the semi-discretization \eqref{eq:semi1}--\eqref{eq:semi2} and the full-discretization \eqref{eq:psifull1}--\eqref{eq:psifull2} with \eqref{eq:wfull} up to the time $T_{\eps}$, respectively.

\subsection{Main results}
Let $w^{[n]}$ and $z^{[n]}$ be the numerical approximations obtained from the Lawson-type exponential integrator (LEI) \eqref{eq:semi1}--\eqref{eq:semi2}, then we have the following improved uniform error bounds for the semi-discretization \eqref{eq:semi1}--\eqref{eq:semi2}.

\begin{theorem}
\label{thm:semi}
Under the  assumption {\rm (A)}, for $0 < \tau_0 \leq 1$ sufficiently small and independent of $\varepsilon$ such that,
 when $0< \tau < \beta \tau_0 $ for a fixed constant $\beta > 0$, we have the following improved error bound
\begin{equation}
\|w(\cdot, t_n) - w^{[n]}\|_{H^1} + \|\partial_t w(\cdot, t_n) - z^{[n]}\|_{L^2} \lesssim \eps^2\tau + \tau_0^{m+1},\quad 0 \leq n \leq \frac{T/\eps^2}{\tau}.
\label{eq:error_semi}
\end{equation}
 In particular, if the exact solution is sufficiently smooth, e.g., $w, \partial_t w \in H^{\infty}$, the last term $\tau_0^{m+1}$ decays exponentially fast and can be ignored practically for small enough $\tau_0$,  and the improved error bound for sufficiently small $\tau$ is
\begin{equation}
\|w(\cdot, t_n) - w^{[n]}\|_{H^1} + \|\partial_t w(\cdot, t_n) - z^{[n]}\|_{L^2} \lesssim \eps^2\tau, \quad 0 \leq n \leq \frac{T/\eps^2}{\tau}.
\label{eq:inf}
\end{equation}
\end{theorem}

Correspondingly, for the full-discretization \eqref{eq:psifull1}--\eqref{eq:psifull2} with \eqref{eq:wfull}, we have the following improved uniform error bounds up to the time $T_{\eps}$.
\begin{theorem}
\label{thm:full}
Under the  assumption {\rm (A)}, there exist $h_0 > 0$ and $0 < \tau_0 < 1$ sufficiently small and independent of $\varepsilon$ such that, for any $0 < \eps \leq 1$, when $0 < h \leq h_0$ and $0< \tau < \beta \tau_0 $ for a fixed constant $\beta > 0$, we have the following improved error bound
\begin{equation}
\|w(\cdot, t_n) - I_M w^n\|_{H^1} + \|\partial_t w(\cdot, t_n) - I_M z^n\|_{L^2} \lesssim h^m + \eps^2\tau + \tau_0^{m+1},\quad 0 \leq n \leq \frac{T/\eps^2}{\tau}.
\label{eq:error_full}
\end{equation}
 In particular, if the exact solution is sufficiently smooth, e.g., $w, \partial_t w \in H^{\infty}$, the improved error bound for sufficiently small $\tau$ is
\begin{equation}
\|w(\cdot, t_n) - I_M w^n\|_{H^1} + \|\partial_t w(\cdot, t_n) - I_M z^n\|_{L^2} \lesssim h^m + \eps^2\tau,\quad 0 \leq n \leq \frac{T/\eps^2}{\tau}.
\end{equation}
\end{theorem}

\begin{remark}
Here, $\tau_0 \in (0, 1)$ is a cut-off parameter introduced in the proof for the improved uniform error bounds. In the numerical analysis, the high frequency Fourier modes $\vert l\vert > 1/\tau_0$ are controlled by the Fourier projection and the requirement $\tau \lesssim \tau_0$ enables the improved error bounds on the low frequency Fourier modes $\vert l\vert \leq 1/\tau_0$, where the constant in front of $\eps^2\tau^2$ depends on $\beta$. Here, $\tau_0$ can be chosen arbitrarily as long as the relation between $\tau$ and $\tau_0$ holds.
\end{remark}

\begin{remark}
In Theorem \ref{thm:semi} and Theorem \ref{thm:full} and the other results in this paper for the one-dimensional problem, we prove the error bounds for $\psi(x, t)$ in $H^1$-norm (i.e., for $w(x, t)$ in $H^1$-norm and $\partial_t w(x, t)$ in $L^2$-norm) due to the fact that $H^r$ is an algebra for $r > d/2$. In  two and three dimensional cases, the corresponding estimates should be in $H^2$-norm, which require higher regularity assumptions of the exact solution.
\end{remark}

\begin{remark}
The LEI-FP method can be extended to numerically solve the oscillatory SGE \eqref{eq:HOE}.  By taking the time step size $\kappa = \eps^2\tau$, the improved error bounds on the LEI-FP method for the long-time dynamics of the SGE \eqref{eq:21} can be extended to the oscillatory SGE \eqref{eq:HOE} up to the fixed time $T$. In 1D,  let $v^n$ and $q^n$ be the numerical approximations of $v(x, s_n)$ and $\partial_s v(x, s_n)$, respectively, and assume the exact solution $v(x, s)$ of the oscillatory SGE \eqref{eq:HOE} satisfies for some $m\ge 0$,
\begin{align*}
& v \in L^{\infty}\left([0, T]; H^{m+1}\right),\quad   \partial_s v \in L^{\infty}\left([0, T]; H^{m}\right), \\
& \|v\|_{L^{\infty}\left([0, T]; H^{m+1}\right)} \lesssim 1,\quad   \|\partial_s v\|_{L^{\infty}\left([0, T]; H^{m}\right)} \lesssim \frac{1}{\eps^2};
\end{align*}
then there exist $h_0 > 0$ and $0 < \kappa_0 < 1$ sufficiently small and independent of $\eps$ such that, for any $0 < \eps \leq 1$, when the mesh size $0 < h \leq h_0$ and the time step $0 < \kappa \leq \beta\eps^2 \kappa_0$ for a fixed constant $\beta > 0$, we have the following improved error bound
\begin{equation}
\left\|v(\cdot, s_n) - I_M v^n\right\|_{H^1} + \eps^2	\left\|\partial_s v(\cdot, s_n) - I_M q^n\right\|_{L^2} \lesssim h^m + \kappa +\kappa^{m+1}_0, \ 0 \leq n \leq \frac{T}{\kappa}. 
\label{eq:HOE_error}
\end{equation}
 In particular, if the exact solution is sufficiently smooth, e.g., $v, \partial_s v \in H^{\infty}$, the improved error bound for sufficiently small $\kappa$ is
\begin{equation}
\|v(\cdot, t_n) - I_M v^n\|_{H^1} + \eps^2\|\partial_t v(\cdot, t_n) - I_M q^n\|_{L^2} \lesssim h^m +\kappa,\quad 0 \leq n \leq \frac{T}{\kappa}.
\end{equation}
From the improved error bound \eqref{eq:HOE_error}, we obtain the temporal error of the LEI-FP method for the oscillatory SGE \eqref{eq:HOE} is independent of $\eps$, but the temporal resolution is still $O(\eps^2)$. In other words, in practical simulations, it needs to choose the time step size $\kappa \lesssim \eps^2$ to obtain the accurate numerical approximation.
\end{remark}

\subsection{Proof for Theorem \ref{thm:semi}}
The assumption (A) is equivalent to the regularity of $\psi(x, t)$ as $
\|\psi\|_{L^{\infty}\left([0, T_{\varepsilon}]; H^{m+1}\right)} \lesssim 1$. Since $f(\phi) = \frac{1}{\eps}\sin(\eps \phi)-\phi$, we can write the function $g(\phi)$ with the $O(\eps^2)$ dominant term as
\begin{equation}
g(\phi) = f\left(\frac{1}{2}(\phi+\overline{\phi})\right) = -\frac{\eps^2}{48}\left(\phi+\overline{\phi}\right)^3 + \eps^4 r(\phi) =: \eps^2h(\phi) + \eps^4 r(\phi).
\end{equation}
Define the function $H$ as
\begin{equation}
H(\phi) = \eps^2i \langle\nabla\rangle^{-1}h(\phi).
\end{equation}
and let
\begin{equation}
H_t: \phi \mapsto e^{-it \langle\nabla\rangle}H\left(e^{it \langle\nabla\rangle}\phi\right), \quad t \in \mathbb{R},
\end{equation}
then we have the following estimates for the local truncation error \cite{BCF,BaoFS}.
\begin{lemma}
\label{lemma:local}
For $0<\eps\le 1$, the local error of the LEI scheme \eqref{eq:semi1} for the relativistic NLSE \eqref{eq:NLS} can be written as ($n = 0, 1, \ldots$)
\begin{equation}
\mathcal{E}^{n}  := \mathcal{L}_{\tau}(\psi(t_n)) - \psi(t_{n+1})  = \mathcal{H}(\psi(t_n))  + \mathcal{R}^n, \quad 
\label{eq:local}
\end{equation}
where
\begin{equation}\label{eq:mH-def}
\mathcal{H}(\psi(t_n)) 	= e^{i\tau\langle\nabla\rangle}\left(\tau H_0(\psi(t_n)) - \int^{\tau}_0  H_{\sigma}(\psi(t_n)) d \sigma \right),
\end{equation}
and the following error bounds  hold under the assumption (A) with $m \geq 0$,
\begin{equation}
\left\|\mathcal{H}(\psi(t_n))\right\|_{H^1}  \lesssim \varepsilon^2\tau^2, \quad \left\|\mathcal{R}^n\right\|_{H^1} \lesssim \varepsilon^4\tau^2.
\label{eq:localbound}
\end{equation}
\end{lemma}
\begin{proof}By the decomposition of the function $g$, we have 
\begin{equation}
F(\phi) = i\langle\nabla\rangle^{-1}g(\phi) = H(\phi) + \eps^4 i\langle\nabla\rangle^{-1}r(\phi).
\end{equation}
Recall \eqref{eq:exact}, we can write the local truncation error as
\begin{align*}
\mathcal{E}^{n}  &=  \tau e^{i\tau\langle\nabla\rangle} F(\psi(t_n))-  \int^{\tau}_0 e^{i(\tau-\sigma)\langle\nabla\rangle} F(\psi(t_n + \sigma)) d \sigma \\
& =  \tau e^{i\tau\langle\nabla\rangle} H(\psi(t_n))-  \int^{\tau}_0 e^{i(\tau-\sigma)\langle\nabla\rangle} H(\psi(t_n + \sigma)) d \sigma+\mathcal{R}^n_1 \\
& =  \tau e^{i\tau\langle\nabla\rangle} H(\psi(t_n)) + \mathcal{R}^n_1 \\
& \quad -  \int^{\tau}_0 e^{i(\tau-\sigma)\langle\nabla\rangle} H\Bigg(e^{i\sigma\langle\nabla\rangle}\psi(t_n) + \int^{\sigma}_0 e^{i(\sigma-\theta)\langle\nabla\rangle} H(\psi(t_n + \theta)) d \theta\Bigg) d \sigma  \\
& = e^{i\tau\langle\nabla\rangle}\left(\tau H(\psi(t_n)) - \int^{\tau}_0  e^{-i\sigma\langle\nabla\rangle} H(e^{i\sigma\langle\nabla\rangle}\psi(t_n)) d \sigma \right)+ \mathcal{R}^n_1+ \mathcal{R}^n_2 \\
& = \mathcal{H}^n(\psi(t_n)) + \mathcal{R}^n,
\end{align*}
where $\mathcal{H}^n(\psi(t_n))$ is defined in \eqref{eq:mH-def} and $\mathcal{R}^n = \mathcal{R}^n_1+\mathcal{R}^n_2$.

Since the operator $e^{it\langle\nabla\rangle}$ is an isometry on $H^1$ and $H^1$ is an algebra in 1D, we obtain that under the assumption (A) with $m \geq 0$,
\begin{equation*}
\left\|\mathcal{H}(\psi(t_n))\right\|_{H^1}  \lesssim \tau^2 \left\|\partial_{\sigma}\left(e^{-i\sigma\langle\nabla\rangle} H(e^{i\sigma\langle\nabla\rangle}\psi(t_n))\right)\right\|_{H^1}\lesssim  \eps^2\tau^2 \left\|\psi(t_n)\right\|_{H^1},
\end{equation*}
and $\left\|\mathcal{R}^n\right\|_{H^1} \lesssim \varepsilon^4\tau^2$, which complete the proof of the error bounds \eqref{eq:localbound}.
\end{proof}

\noindent
{\emph{Proof for Theorem \ref{thm:semi}}} Under the assumption (A), we will take the induction argument to prove that there exists $\tau_{c}>0$ such that for $0<\tau<\tau_c$ the following estimates hold
\begin{equation}
\label{eq:semi_induc}
\|e^{[n]}\|_{H^1}\leq C(\eps^2\tau+\tau_0^{m+1}),\quad \|\psi^{[n]}\|_{H^1}\leq M+1,\quad 0\leq n\leq \frac{T/\eps^2}{\tau},
\end{equation}
where $M=\left\|\psi\right\|_{L^{\infty}([0, T_\eps]; H^1)}$ and $C>0$ is independent of $n$,  $\eps$ and $\tau$. Since $\psi^{[0]}=\psi_0$, the case $n=0$ is obvious. Assume the error bound \eqref{eq:semi_induc} holds true for all $0 \le n \le p \le \frac{T/\eps^2}{\tau} - 1$, then we are going to prove the case $n = p+1$. 

For the numerical approximation $\psi^{[n]}$ obtained by the LEI \eqref{eq:semi1}, introduce the error function 
\begin{equation}
e^{[n]}:=e^{[n]}(x) = \psi^{[n]}-\psi(t_n), \quad 0 \le n \le \frac{T/\eps^2}{\tau},
\end{equation}
 then from \eqref{eq:semi1} and \eqref{eq:local}, we have the following error equation
\begin{align}
e^{[n+1]} & =  e^{i\tau\langle\nabla\rangle}e^{[n]} + W^n + \mathcal{E}^{n},\quad 0 \leq n \leq \frac{T/\eps^2}{\tau},
\label{eq:eg_semi}
\end{align}
where $W^n := W^n(x)$ is given by
\begin{equation*}
W^n(x) = \tau e^{i\tau\langle\nabla\rangle} \left(F(\psi^{[n]}) -F(\psi(t_n))\right).
\end{equation*}
Under the assumption (A) and \eqref{eq:semi_induc} for $n \le p$ , we have
\begin{equation}
\left\|W^n\right\|_{H^1}\lesssim \tau \left\|\left(F(\psi^{[n]}) -F(\psi(t_n))\right)\right\|_{H^1} \lesssim\eps^2\tau\left\|e^{[n]}\right\|_{H^1},\quad n \le p.
\label{eq:Wb}
\end{equation}
According to the error function \eqref{eq:eg_semi}, we obtain
\begin{equation}
e^{[n+1]} = e^{i(n+1)\tau\langle\nabla\rangle}e^{[0]} + \sum\limits_{k=0}^n e^{i(n-k)\tau\langle\nabla\rangle}\Big(W^k + \mathcal{E}^k\Big).
\end{equation}
Since $e^{[0]} = 0$, combining with \eqref{eq:local}, \eqref{eq:localbound} and  \eqref{eq:Wb}, we have the following estimate
\begin{equation}
\label{eq:egrow}
\left\|e^{[n+1]}\right\|_{H^1}	\lesssim \varepsilon^2\tau + \varepsilon^2 \tau  \sum_{k=0}^n \left\|e^{[k]}\right\|_{H^1} + \left\| \sum\limits_{k=0}^n e^{i(n-k)\tau\langle\nabla\rangle}\mathcal{H}(\psi(t_k))\right\|_{H^1}.
\end{equation}

The regularity compensation oscillation (RCO) technique has been introduced to establish the improved uniform error bounds of the time-splitting method for the (nonlinear) Schr\"odinger equation, nonlinear Klein--Gordon equation and Dirac equation \cite{BCF,BCF2,BFYIN}. Here, for the Lawson-type exponential integrator scheme, we apply the RCO technique to deal with the last term in the RHS of \eqref{eq:egrow} and obtain the improved uniform error bound. From the relativistic NLSE \eqref{eq:NLS}, we find that $\partial_t \psi(x, t) - i\langle\nabla\rangle \psi(x, t) = F(\phi)=O(\eps^2)$. Thus, we introduce the `twisted variable'
 \begin{equation}
 \phi(x, t) = e^{-it\langle\nabla\rangle}\psi(x, t), \quad t \geq 0,
 \label{eq:twist_def}
 \end{equation}
which satisfies the equation $\partial_t\phi(x, t)=e^{-it\langle\nabla\rangle}F(e^{it\langle\nabla\rangle}\phi(x, t))$. Under the assumption (A), we have $\|\phi\|_{L^{\infty}([0, T_{\eps}]; H^{m+1})}\lesssim1$ and $\left\|\partial_t \phi\right\|_{L^{\infty}([0, T_{\eps}]; H^{m+1})} \lesssim \eps^2$ with
\begin{equation}
\|\phi(t_{n}) - \phi(t_{n-1})\|_{H^{m+1}} \lesssim \varepsilon^2 \tau, \quad 0 \le n \le \frac{T/\eps^2}{\tau}.
\label{eq:twist}
\end{equation}

{\bf Step 1}. Choose the cut-off parameter on the Fourier modes. Let $\tau_0 \in (0, 1)$ and choose $M_0 = 2\lceil 1/\tau_0\rceil \in \mathbb{Z}^+$ ($\lceil \cdot\rceil$ is the ceiling function) with $1/\tau_0 \leq M_0/2 < 1 + 1/\tau_0$.  Under the assumption (A) and the properties of operators $e^{-it\langle\nabla\rangle}$ and  $\langle\nabla\rangle^{-1}$, we have 
\begin{equation}
\|P_{M_0}\mathcal{H}(e^{i t_k\langle\nabla\rangle}(P_{M_0}\phi(t_k)))-
\mathcal{H}(e^{i t_k\langle\nabla\rangle}\phi(t_k))\|_{H^1} \lesssim \eps^2\tau\tau_0^{m+1}.
\end{equation}
Combining above estimates, we obtain for $n\leq p$,

\begin{equation}
\left\|e^{[n+1]}\right\|_{H^1}  \lesssim \tau_0^{m+1} +\varepsilon^2\tau +\varepsilon^2\tau\sum_{k=0}^n\left\|e^{[k]}\right\|_{H^1} +\left\|\mathcal{J}^n\right\|_{H^1}, \label{eq:final2}
\end{equation}
where
\begin{equation}
 \mathcal{J}^n = \sum\limits_{k=0}^n e^{-i(k+1)\tau\langle\nabla\rangle}P_{M_0}\mathcal{H}(e^{i t_k\langle\nabla\rangle}(P_{M_0}\phi(t_k))).
 \label{eq:def_J}
\end{equation}

{\bf Step 2}. Analyze the low Fourier modes term $\mathcal{J}^n$. Recalling the function $ H(\phi)$, we have the decomposition
\begin{equation}
H(\phi)=\sum_{q=1}^4 H^{\{q\}}(\phi),\quad H^{\{q\}}(\phi)= -\frac{\eps^2}{48} i\langle\nabla\rangle^{-1}h^{\{q\}}(\phi),\quad q=1, 2, 3, 4,
\end{equation}
with $ h^{\{1\}}(\phi)=\phi^3, h^{\{2\}}(\phi)=3\bar{\phi}\phi^2, h^{\{3\}}(\phi)=3\bar{\phi}^2\phi, h^{\{4\}}(\phi)=\bar{\phi}^3$. For $\sigma \in \mathbb{R}$ and $q = 1, 2, 3, 4$, introducing $ H_{\sigma}^{\{q\}}(\psi(t_k)) = e^{-i\sigma\langle\nabla\rangle} H^{\{q\}}(e^{i\sigma\langle\nabla\rangle}\psi(t_k))$ and
\begin{equation}
\mathcal{H}^{\{q\}}(\psi(t_k)) 	= e^{i\tau\langle\nabla\rangle}\left(\tau H_0^{\{q\}}(\psi(t_k)) - \int^{\tau}_0  H_{\sigma}^{\{q\}}(\psi(t_k)) d \sigma \right),
\end{equation}
we have for $q = 1, 2, 3, 4$,
\begin{equation}\label{eq:Ln1}
\mathcal{J}^n=\sum_{q=1}^4\mathcal{J}_q^n,\quad \mathcal{J}_q^n = \sum\limits_{k=0}^n e^{-i(k+1)\tau\langle\nabla\rangle}P_{M_0}\mathcal{H}^{\{q\}}(e^{i t_k\langle\nabla\rangle}(P_{M_0}\phi(t_k))).
\end{equation}
Since the estimates for $\mathcal{J}^n_q$ ($q= 1, 2, 3, 4$) are indeed same, we only show the estimates for $\mathcal{J}^n_1$ in details as an example. For $l\in\mathcal{T}_{M_0}$, define the index set $\mathcal{I}_l^{M_0}$ associated to $l$ as
\begin{equation}
\mathcal{I}_l^{M_0}=\left\{(l_1, l_2, l_3)~\vert~ \ l_1+l_2+l_3=l,\ l_1, l_2, l_3\in\mathcal{T}_{M_0}\right\},
\end{equation}
then we have the following expansion
\begin{align*}
& e^{-it_{k+1}\langle\nabla\rangle}P_{M_0}( e^{i\tau\langle\nabla\rangle}H^{\{1\}}_{\sigma}(e^{i t_k\langle\nabla\rangle}P_{M_0}\phi(t_k))) \\
&=\quad -\eps^2\sum\limits_{l\in\mathcal{T}_{M_0}}\sum\limits_{(l_1, l_2, l_3)\in\mathcal{I}_l^{M_0}}\frac{i}{48\delta_l} \mathcal{G}^k_{l, l_1, l_2, l_3}(\sigma) e^{i\mu_l(x-a)},
\end{align*}
where the coefficients  $\mathcal{G}^k_{l,l_1,l_2,l_3}(\sigma)$ are functions of $sigma\in\mathbb{R}$ defined as
\begin{equation}
\mathcal{G}^k_{l,l_1,l_2,l_3}(\sigma)
= e^{-i(t_k+\sigma)\delta_{l,l_1,l_2,l_3}}\widehat{\phi_{l_1}}(t_k)\widehat{\phi}_{l_2}(t_k)\widehat{\phi}_{l_3}(t_k)
\label{eq:mGdef}
\end{equation}
with $\delta_{l,l_1,l_2,l_3} = \delta_l -\delta_{l_1} -\delta_{l_2} - \delta_{l_3}$. Thus, we have
\begin{equation}\label{eq:remainder-dec}
\mathcal{J}_1^n  = -\frac{i\eps^2}{48}\sum\limits_{k=0}^n
\sum\limits_{l\in\mathcal{T}_{M_0}}\sum\limits_{(l_1,l_2,l_3)\in\mathcal{I}_l^{M_0}}\frac{1}{\delta_l}\Lambda^k_{l, l_1, l_2, l_3}e^{i\mu_l(x-a)},
\end{equation}
where
\begin{align}
\Lambda^k_{l,l_1,l_2,l_3} & = - \tau \mathcal{G}^k_{l, l_1, l_2, l_3}(0) +\int_0^\tau\mathcal{G}^k_{l,l_1,l_2,l_3}(\sigma) d \sigma\\
 & = r_{l, l_1, l_2, l_3}e^{-it_k\delta_{l, l_1, l_2, l_3}}c^k_{l, l_1, l_2, l_3},
\label{eq:Lambdak}
\end{align}
with coefficients $c^k_{l, l_1, l_2, l_3}$ and $r_{l, l_1, l_2, l_3}$ given by
\begin{align}
c^k_{l, l_1, l_2, l_3}=& \ \widehat{\phi}_{l_1}(t_k) \widehat{\phi}_{l_2}(t_k)\widehat{\phi}_{l_3}(t_k),\label{eq:calFdef}\\
r_{l, l_1, l_2, l_3}= & \ -\tau +\int_0^\tau e^{-i\sigma\delta_{l, l_1, l_2, l_3}} d \sigma = O\left(\tau^2 \delta_{l, l_1, l_2, l_3}\right).\label{eq:rest}
\end{align}
We only need  to consider the case $\delta_{l, l_1, l_2, l_3}\neq 0$ as $r_{l, l_1, l_2, l_3} = 0$ if $\delta_{l, l_1, l_2, l_3}= 0$. For $l\in\mathcal{T}_{M_0}$ and $(l_1, l_2, l_3)\in\mathcal{I}_l^{M_0}$, we have
\begin{equation}
\vert\delta_{l, l_1, l_2, l_3}\vert\leq 4\delta_{M_0/2}= 4 \sqrt{1+\mu_{M_0/2}^2} <  4\sqrt{1+\frac{4\pi^2(1+\tau_0)^2}{\tau_0^2(b-a)^2}},
\end{equation}
which implies
\begin{equation}
 \frac{\tau}{2}\vert\delta_{l,l_1,l_2,l_3}\vert \leq \alpha\pi,
 \end{equation}
if $0 < \tau \leq \alpha \frac{\pi (b-a)\tau_0}{2\sqrt{\tau_0^2(b-a)^2+4\pi^2(1+\tau_0)^2}}:= \beta \tau_0$.
Denoting $S^n_{l,l_1,l_2,l_3}=\sum_{k=0}^n e^{-it_k\delta_{l,l_1,l_2,l_3}}$ ($n\ge0$), for  $0< \tau \leq \beta\tau_0$, we then obtain
\begin{equation}\label{eq:Sbd}
\vert S^n_{l,l_1,l_2,l_3}\vert\leq \frac{1}{\vert\sin(\tau \delta_{l,l_1,l_2,l_3}/2)\vert}\leq\frac{C}{\tau\vert\delta_{l,l_1,l_2,l_3}\vert},\quad C = \frac{2\alpha\pi}{\sin(\alpha \pi)},\quad \forall n\ge0.
\end{equation}
Using summation-by-parts, \eqref{eq:Lambdak} implies
\begin{equation}
\sum_{k=0}^n\Lambda^k_{l,l_1,l_2,l_3}=r_{l,l_1,l_2,l_3}\big[\sum_{k=0}^{n-1}S^k_{l,l_1,l_2,l_3} (c^k_{l,l_1,l_2,l_3}-c^{k+1}_{l,l_1,l_2,l_3})+S^n_{l,l_1,l_2,l_3}c^n_{l,l_1,l_2,l_3}\big],
\label{eq:lambdasum}
\end{equation}
with
\begin{align}
&c^k_{l,l_1,l_2,l_3}-c^{k+1}_{l,l_1,l_2,l_3}\nn\\
& = (\widehat{\phi}_{l_1}(t_k)-\widehat{\phi}_{l_1}(t_{k+1})) \widehat{\phi}_{l_2}(t_k)\widehat{\phi}_{l_3}(t_k) + \widehat{\phi}_{l_1}(t_{k+1})(\widehat{\phi}_{l_2}(t_k)-\widehat{\phi}_{l_2}(t_{k+1}))\widehat{\phi}_{l_3}(t_k)\nn \\
&\;\;\;\;\; + \widehat{\phi}_{l_1}(t_{k+1}) \widehat{\phi}_{l_2}(t_{k+1})(\widehat{\phi}_{l_3}(t_k)-\widehat{\phi}_{l_3}(t_{k+1})).\label{eq:cksum}
\end{align}

Combining \eqref{eq:rest}, \eqref{eq:Sbd}, \eqref{eq:lambdasum} and \eqref{eq:cksum}, we have
\begin{align}
\left\vert\sum_{k=0}^n\Lambda^k_{l,l_1,l_2,l_3}\right\vert
\lesssim & \  \tau\sum\limits_{k=0}^{n-1}\bigg(
\left\vert\widehat{\phi}_{l_1}(t_k)-\widehat{\phi}_{l_1}(t_{k+1})\right\vert\left\vert\widehat{\phi}_{l_2}(t_k)\right\vert \left\vert\widehat{\phi}_{l_3}(t_k)\right\vert\nn\\
&\ +\left\vert\widehat{\phi}_{l_1}(t_{k+1})\right\vert\left\vert\widehat{\phi}_{l_2}(t_k)-\widehat{\phi}_{l_2}(t_{k+1})\right\vert\left\vert\widehat{\phi}_{l_3}(t_k)\right\vert \nn\\
&\ +\left\vert\widehat{\phi}_{l_1}(t_{k+1})\right\vert\left\vert\widehat{\phi}_{l_2}(t_{k+1})\right\vert \left\vert\widehat{\phi}_{l_3}(t_k)-\widehat{\phi}_{l_3}(t_{k+1})\right\vert\bigg)\nn\\
&\ + \tau \left\vert\widehat{\phi}_{l_1}(t_n)\right\vert\left\vert\widehat{\phi}_{l_2}(t_n)\right\vert \left\vert\widehat{\phi}_{l_3}(t_n)\right\vert.\label{eq:sumlambda}
\end{align}
Based on \eqref{eq:remainder-dec} and \eqref{eq:sumlambda}, we have
\begin{align}
&\left\|\mathcal{J}_1^n\right\|^2_{H^1}  \nn \\
& = \ \eps^4
\sum\limits_{l\in\mathcal{T}_{M_0}}\left\vert\sum\limits_{(l_1,l_2,l_3)\in\mathcal{I}_l^{M_0}}\sum\limits_{k=0}^n\Lambda^k_{l,l_1,l_2,l_3}\right\vert^2\nn \\
& \lesssim \ \eps^4\tau^2
\bigg\{\sum_{l\in\mathcal{T}_{M_0}}\bigg(\sum\limits_{(l_1,l_2,l_3)\in\mathcal{I}_l^{M_0}}\left\vert\widehat{\phi}_{l_1}(t_n)\right\vert\left\vert\widehat{\phi}_{l_2}(t_n)\right\vert \left\vert\widehat{\phi}_{l_3}(t_n)\right\vert\bigg)^2\nn \\
&\;\;\;\; + n \sum\limits_{k=0}^{n-1}\sum_{l\in\mathcal{T}_{M_0}}\bigg[\bigg(\sum\limits_{(l_1,l_2,l_3)\in\mathcal{I}_l^{M_0}}
\left\vert\widehat{\phi}_{l_1}(t_k)-\widehat{\phi}_{l_1}(t_{k+1})\right\vert\left\vert\widehat{\phi}_{l_2}(t_k)\right\vert \left\vert\widehat{\phi}_{l_3}(t_k)\right\vert\bigg)^2\nn\\
& \;\;\;\; +\bigg(\sum\limits_{(l_1,l_2,l_3)\in\mathcal{I}_l^{M_0}}
\left\vert\widehat{\phi}_{l_1}(t_{k+1})\right\vert\left\vert\widehat{\phi}_{l_2}(t_k)-\widehat{\phi}_{l_2}(t_{k+1})\right\vert \left\vert\widehat{\phi}_{l_3}(t_k)\right\vert\bigg)^2\nn\\
& \;\;\;\; +\bigg(\sum\limits_{(l_1,l_2,l_3)\in\mathcal{I}_l^{M_0}}
\left\vert\widehat{\phi}_{l_1}(t_{k+1})\right\vert\left\vert\widehat{\phi}_{l_2}(t_{k+1})\right\vert \left\vert\widehat{\phi}_{l_3}(t_k)-\widehat{\phi}_{l_3}(t_{k+1})\right\vert\bigg)^2\bigg]\bigg\}.
\label{eq:sumlambda-2}
\end{align}
Expanding $(\phi(x))^3=\sum\limits_{l\in\mathbb{Z}}\sum\limits_{l_1+l_2+l_3=l, l_j\in\mathbb{Z}}\left\vert\widehat{\phi}_{l_1}(t_n)\right\vert\left\vert\widehat{\phi}_{l_2}(t_n)\right\vert\left\vert\widehat{\phi}_{l_3}(t_n)\right\vert e^{i\mu_l(x-a)}$, we obtain
\begin{equation*}
\sum_{l\in\mathcal{T}_{M_0}} \bigg(\sum\limits_{(l_1,l_2,l_3)\in\mathcal{I}_l^{M_0}}\left\vert\widehat{\phi}_{l_1}(t_n)\right\vert\left\vert\widehat{\phi}_{l_2}(t_n)\right\vert \left\vert\widehat{\phi}_{l_3}(t_n)\right\vert\bigg)^2\leq \left\|\phi^3(x)\right\|_{L^2}^2 \lesssim \left\|\phi(x)\right\|_{H^1}^6 \lesssim1.
\end{equation*}
Thus, in light of \eqref{eq:twist}, we estimate each term in \eqref{eq:sumlambda-2} similarly as
\begin{align}
 \|\mathcal{J}_1^n\|_1^2& \ \lesssim \eps^4\tau^2 \bigg[\left\|\phi(t_n)\right\|_{H^1}^6+n\sum\limits_{k=0}^{n-1}
\left\|\phi(t_k) - \phi(t_{k+1})\right\|_{H^1}^2(\left\|\phi(t_k)\right\|_{H^1}+ \left\|\phi(t_{k+1})\right\|_{H^1})^4\bigg]\nn \\
&\ \lesssim  \eps^4\tau^2 + n^2\eps^4\tau^2 (\eps^2\tau)^2
\lesssim \eps^4\tau^2,\quad 0\le n\leq p.
\label{eq:est-l2}
\end{align}
In the same process, we have the estimates for $\mathcal{J}_q^n$ with $q =2, 3, 4$. Substituting the estimates for $\mathcal{J}^n$ into \eqref{eq:final2},  we have
\begin{equation}
\left\|e^{[n+1]}\right\|_{H^1} \lesssim \tau_0^{m+1} +\eps^2\tau+\eps^2\tau\sum_{k=0}^n\left\|e^{[k]}\right\|_{H^1},\quad 0\le n\leq p.
\end{equation}
By Gronwall inequality, we obtain
\begin{equation}
\left\|e^{[n+1]}\right\|_{H^1} \lesssim \eps^2\tau + \tau_0^{m+1},\quad  0\le n\leq p,
\end{equation}
which shows that the first inequality in \eqref{eq:semi_induc} holds for $n = p+1$. Then, it is easy to check that
\begin{equation}
\left\|\psi^{[p+1]}\right\|_{H^1}	 \le \left\|\psi(t_{p+1})\right\|_{H^1}	+ \left\|e^{[p+1]}\right\|_{H^1} \leq M+1,
\end{equation}
which means that the second inequality in  \eqref{eq:semi_induc} also holds for $n = p+1$ and finishes the induction process. In  view of \eqref{eq:wz} and \eqref{eq:semi2}, the proof for the improved error bound \eqref{eq:error_semi} is completed.

\begin{remark}
If we directly apply the error bounds for the local truncation error \eqref{eq:localbound} and Gronwall inequality in the error growth \eqref{eq:egrow}, we can only get the uniform error bound $\left\|e^{[n+1]}\right\|_{H^1} \lesssim \tau$ for $0 \le n \le \frac{T/\eps^2}{\tau}-1$. Here, we exploit the RCO technique to establish the improved uniform error bound \eqref{eq:error_semi}, but there is no additional requirement for the exact solution of the SGE \eqref{eq:21}.
\end{remark}

\subsection{Proof for Theorem \ref{thm:full}}
In our previous work for the NKGE \cite{BCF}, the improved error bound for the full-discretization is mainly based on the error splitting approach, i.e., we first proof the semi-discrete-in-time error bound \eqref{eq:error_semi}, and then compare the difference between $\psi^{[n]}$ and $\psi^n$. To obtain the convergence order $h^m$ in space, we need the bound of $\|\psi^{[n]}\|_{H^m}$, which is generally not available under the assumption (A) for arbitrary $T > 0$ in the nonlinear equation. Following the classical arguments of Gronwall type, we can show $\|\psi^{[n]}\|_{H^m}$ is bounded for certain $n\tau\leq T_0/\eps^2$, and $T_0>0$ is determined by $\|\psi_0\|_{H^m}$ only, which is not suitable to derive the estimates over the whole interval $[0,T/\eps^2]$ in the assumption (A). Here, we will directly prove the error bound \eqref{eq:error_full} for the LEI-FP method \eqref{eq:psifull1}--\eqref{eq:psifull2} without comparing with the semi-discretization-in-time. Similar to the proof of Lemma \ref{lemma:local}, we have following result for the local truncation error for the LEI-FP method \eqref{eq:psifull1}--\eqref{eq:psifull2} and we omit the details here for brevity.

\begin{lemma}
\label{lemma:full_local}
For $0 < \eps \le 1$, the local truncation error of the  LEI-FP method \eqref{eq:psifull1}--\eqref{eq:psifull2} for the relativistic NLSE \eqref{eq:NLS} can be written as
\begin{equation*}
\overline{\mathcal{E}}^n := P_M\mathcal{L}_{\tau}(P_M\psi(t_{n})) - P_M\psi(t_{n+1}) = P_M \mathcal{H}(P_M\psi(t_n)) +\mathcal{Y}^n,\quad 0 \leq n \leq \frac{T/\eps^2}{\tau}-1,
\end{equation*}
where the following error bounds hold under the assumption (A),
\begin{equation}
\left\|\mathcal{H}(P_M\psi(t_n))\right\|_{H^1} \lesssim \eps^2\tau^2, \quad \left\|\mathcal{Y}^n\right\|_{H^1}	\lesssim \eps^4\tau^2 + \eps^2\tau h^m.
\label{eq:full_lb}
\end{equation}
\end{lemma}

\noindent
{\emph{Proof for Theorem \ref{thm:full}}}  By the definition of $\psi$ in \eqref{eq:psi}, it suffices to prove
\begin{equation}
\|\psi(\cdot, t_n) - I_M \psi^n\|_{H^1} \lesssim h^m + \eps^2\tau + \tau_0^{m+1},\quad 0 \leq n \leq \frac{T/\eps^2}{\tau}.
\label{eq:error_full_psi}
\end{equation}

By the standard Fourier projection and interpolation results, we have $\left\|\psi(t_n) - I_M \psi^n\right\|_{H^1} \lesssim \left\|I_M \psi^n - P_M\psi(t_n)\right\|_{H^1} + h^{m}$, which means that it just needs to consider the growth of the error function $e^n = I_M \psi^n - P_M\psi(t_n)\in Y_M$. For $0 \leq n \leq \frac{T/\eps^2}{\tau}-1$, we have
\begin{equation}
e^{n+1}	= I_M \psi^{n+1} - P_M \mathcal{L}_{\tau}(P_M\psi(t_n)) + \overline{\mathcal{E}}^n =  e^{i\tau\langle\nabla\rangle}e^{n}+{Z}^n + \overline{\mathcal{E}}^n,
\label{eq:etg_nl}
\end{equation}
where $Z^n\in Y_M$ is given by
\begin{equation*}
Z^n =  \tau e^{i\tau\langle\nabla\rangle}\left(I_M F(I_M\psi^n) - P_M F(P_M\psi(t_n))\right).
\end{equation*}
Similar to the proof for the semi-discretization, we also apply the induction argument to prove that there exist $h_{c}>0$ and $\tau_{c}>0$ such that for $0<h<h_c$ and $0<\tau<\tau_c$ we have the following estimates
\begin{equation}
\label{eq:full_induc}
\|e^n\|_{H^1}\leq C(h^m+\eps^2\tau+\tau_0^{m+1}),\ \|I_M\psi^n\|_{H^1}\leq M+1,\quad 0\leq n\leq \frac{T/\varepsilon^2}{\tau},
\end{equation}
where $M=\left\|\psi\right\|_{L^{\infty}([0, T_\eps]; H^1)}$, and $C>0$ is independent of $n$, $h$, $\varepsilon$ and $\tau$.

For $n=0$, \eqref{eq:full_induc} holds for sufficiently small $0 < h < h_1$ with $h_1 >0$ by the standard Fourier interpolation result, i.e. $\|e^0\|_{H^1}\leq C_1h^m$ and $ \|I_M\psi^0\|_{H^1}\leq M+1$. Assume \eqref{eq:full_induc} holds for $0 < n \le p \le \frac{T/\eps^2}{\tau}-1$, and we will prove it holds for the case $n = p+1$. Recalling the definition of $Z^n$, replacing $P_M$ by $I_M$, we have for $0 < h < h_2$ and $0<\tau<\tau_2$, 
\begin{equation}
\left\|Z^n\right\|_{H^1} \lesssim \eps^2 \tau \left( h^m +\|e^n\|_{H^1}\right),\quad 0 \leq n\leq p.
\label{eq:Z_b}
\end{equation}
Combining  the estimates \eqref{eq:full_lb} and \eqref{eq:Z_b},  we get for $0\le n \le p$,
\begin{equation}
\|e^{n+1}\|_{H^1} \lesssim  \ h^m+\eps^2\tau+\varepsilon^2\tau\sum_{k=0}^n\|e^k\|_{H^1}+\left\| \sum\limits_{k=0}^n e^{i(n-k)\tau\langle\nabla\rangle}\mathcal{H}(\psi(t_k))\right\|_{H^1},
\label{eq:final}
\end{equation}
where $\mathcal{H}(\cdot)$ is defined in Lemma \ref{lemma:local}.
Following the proof for the semi-discretization, introducing a cut-off parameter $\tau_0 \in (0, 1)$  and $M_0 = 2\lceil 1/\tau_0 \rceil \in \mathbb{Z}^+$ ($\lceil \cdot \rceil$ is the ceiling function) with $1/\tau_0 \leq M_0/2 < 1+ 1/\tau_0$. Under the assumption (A), replacing $P_M$ by $P_{M_0}$ in \eqref{eq:final}, we can derive
\begin{equation}
\left\|e^{n+1}\right\|_{H^1}\lesssim h^m + \tau_0^{m+1} + \eps^2\tau + \eps^2\tau \sum_{k=0}^{n} \left\|e^{k}\right\|_{H^1} + \left\| \mathcal{J}^n\right\|_{H^1},
\end{equation}
where $\mathcal{J}^n$ is defined in \eqref{eq:def_J}. Recalling the estimate \eqref{eq:est-l2} for  $\mathcal{J}^n$, we obtain
\begin{equation*}
\left\|e^{n+1}\right\|_{H^1}\lesssim h^m + \tau_0^{m+1} + \eps^2\tau + \eps^2\tau \sum_{k=0}^{n} \left\|e^{k}\right\|_{H^1}, \quad 0 \leq  n \leq p.
\end{equation*}
The discrete Gronwall inequality implies
\begin{equation}
\left\|e^{n+1}\right\|_{H^1}\lesssim h^m + \tau_0^{m+1} + \eps^2\tau, \quad 0 \leq  n \leq p,
\end{equation}
and the first inequality in \eqref{eq:full_induc} holds for $n = p+1$. Subsequently, it leads to
\begin{equation}
\left\|\psi^{p+1}\right\|_{H^1} \le \left\|\psi(t_{p+1})\right\|_{H^1}	+ \left\|e^{p+1}\right\|_{H^1} \leq M + 1,
\end{equation}
which indicates the second inequality in  \eqref{eq:full_induc} for $n = p+1$. The induction for the proof of \eqref{eq:error_full_psi} is completed, which implies the improved error bound \eqref{eq:error_full}.

\section{Numerical results}
In this section, we present some numerical examples to illustrate the efficiency of the LEI-FP method and confirm the improved error bounds.

\subsection{Long-time dynamics in 1D}
We present the numerical result for the long-time dynamics of the SGE \eqref{eq:WNE} in 1D to confirm the improved uniform error bound \eqref{eq:error_full}. We choose the initial data as
\begin{equation}
u_0(x) = \frac{2}{1+\cos^2(x)},\quad u_1(x) = \frac{1}{2+\sin(x)}, \quad x \in\Omega= (0, 2\pi). 	
\label{eq:initial}
\end{equation}
The numerical `exact' solution is obtained numerically by the LEI-FP method \eqref{eq:psifull1}--\eqref{eq:psifull2} with \eqref{eq:wfull} for the SGE \eqref{eq:WNE} in 1D with a very small time step $\tau_e = 10^{-4}$ and a very fine mesh size $h_e = \pi/64$. To quantify the error, we introduce the following error function
\begin{equation}
e(t_n) = \left\|w(x, t_n) - I_N w^n\right\|_{H^1} + \left\|\partial_t w(x, t_n) - I_N z^n\right\|_{L^2} .
\end{equation}
The errors are displayed at $t = 1/\eps^2$ with different $\eps$. For spatial errors, we choose a very small time step $\tau_e = 10^{-4}$ such that the temporal errors can be neglected. For temporal errors, we choose a very fine mesh size $h_e = \pi/64$ such that the spatial errors are ignorable. Fig. \ref{fig:1D_h} and Fig. \ref{fig:1D_tau} show the long-time spatial and temporal errors of the the LEI-FP method \eqref{eq:psifull1}--\eqref{eq:psifull2} with \eqref{eq:wfull} for the SGE \eqref{eq:WNE} in 1D with different $\eps$, respectively.

\begin{figure}[ht!]
\begin{minipage}{0.49\textwidth}
\centerline{\includegraphics[width=6.5cm,height=5.5cm]{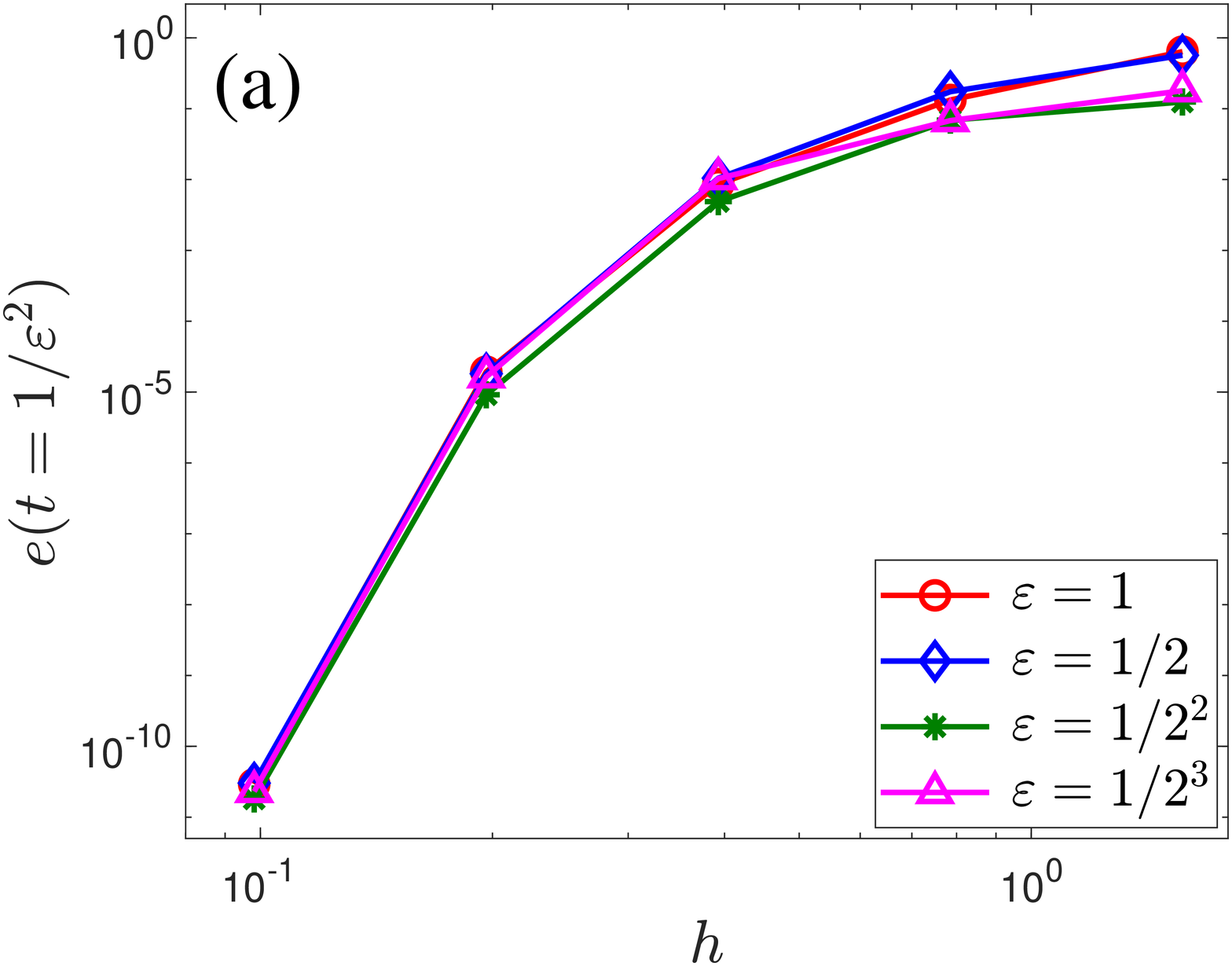}}
\end{minipage}
\begin{minipage}{0.49\textwidth}
\centerline{\includegraphics[width=6.5cm,height=5.5cm]{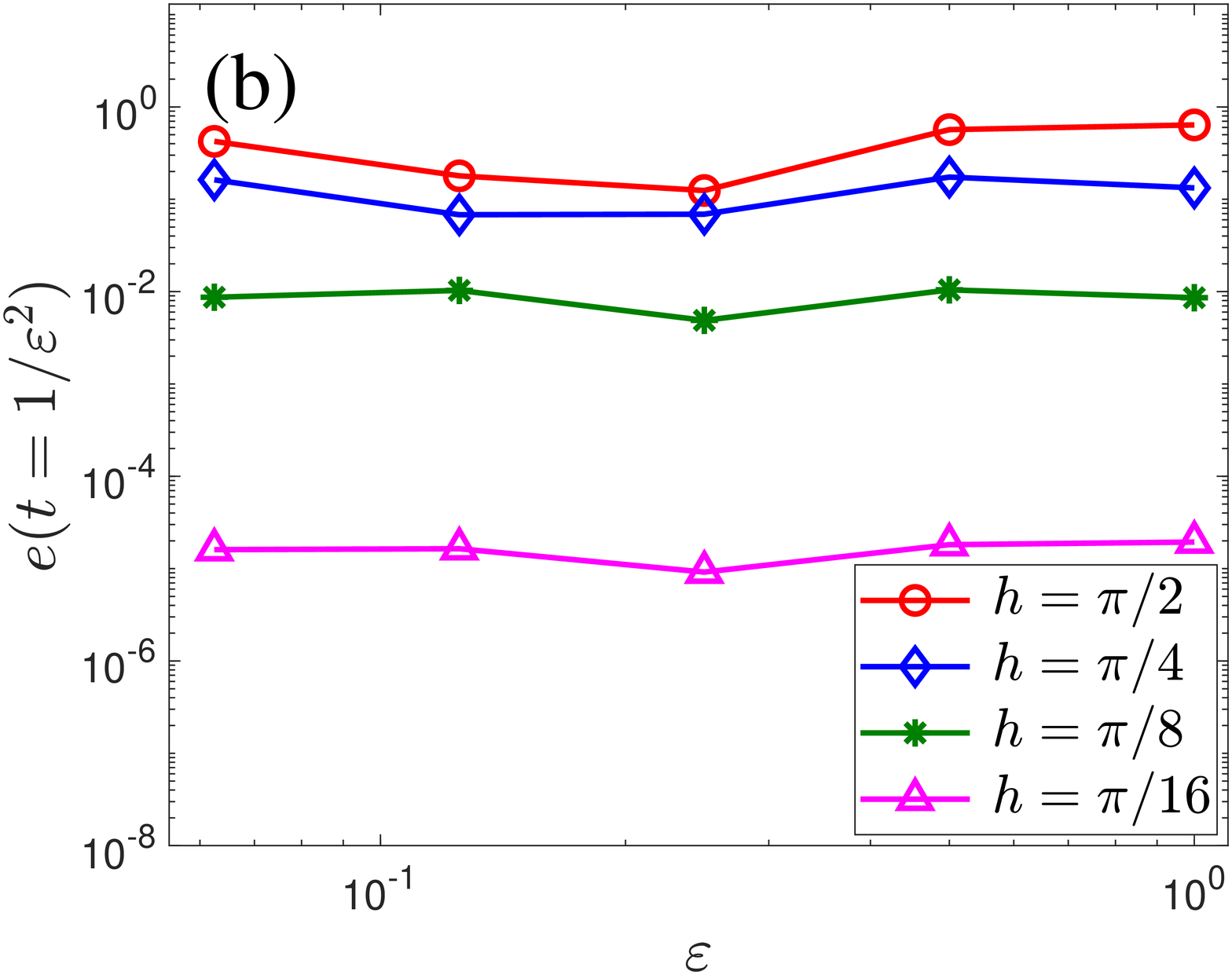}}
\end{minipage}
\caption{Long-time spatial errors of the LEI-FP method \eqref{eq:psifull1}--\eqref{eq:psifull2} with \eqref{eq:wfull} for the SGE \eqref{eq:WNE} in 1D at $t = 1/\eps^2$: (a) convergence rates in $h$, and (b) convergences rate in $\eps$.}
\label{fig:1D_h}
\end{figure}

\begin{figure}[ht!]
\begin{minipage}{0.49\textwidth}
\centerline{\includegraphics[width=6.5cm,height=5.5cm]{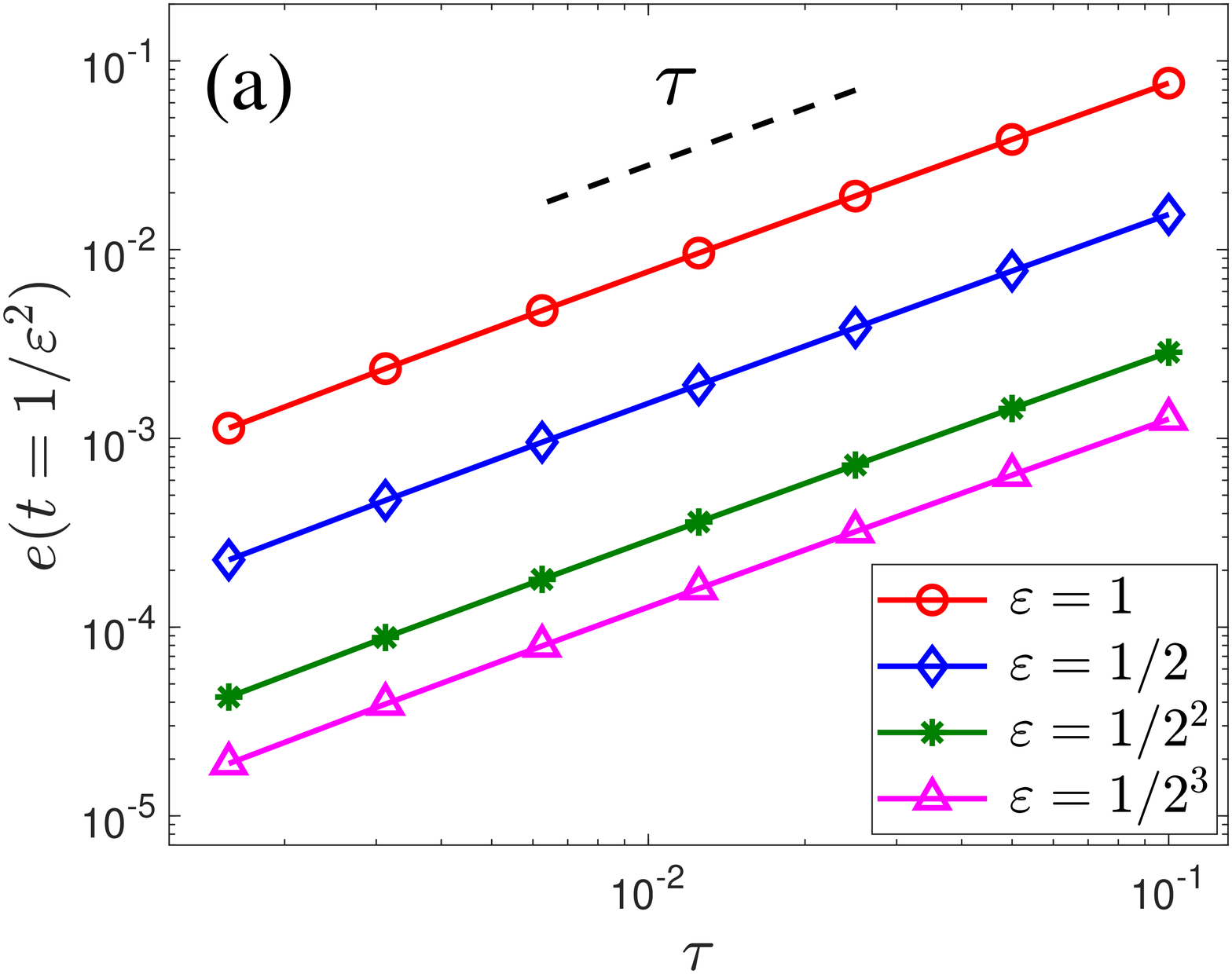}}
\end{minipage}
\begin{minipage}{0.49\textwidth}
\centerline{\includegraphics[width=6.5cm,height=5.5cm]{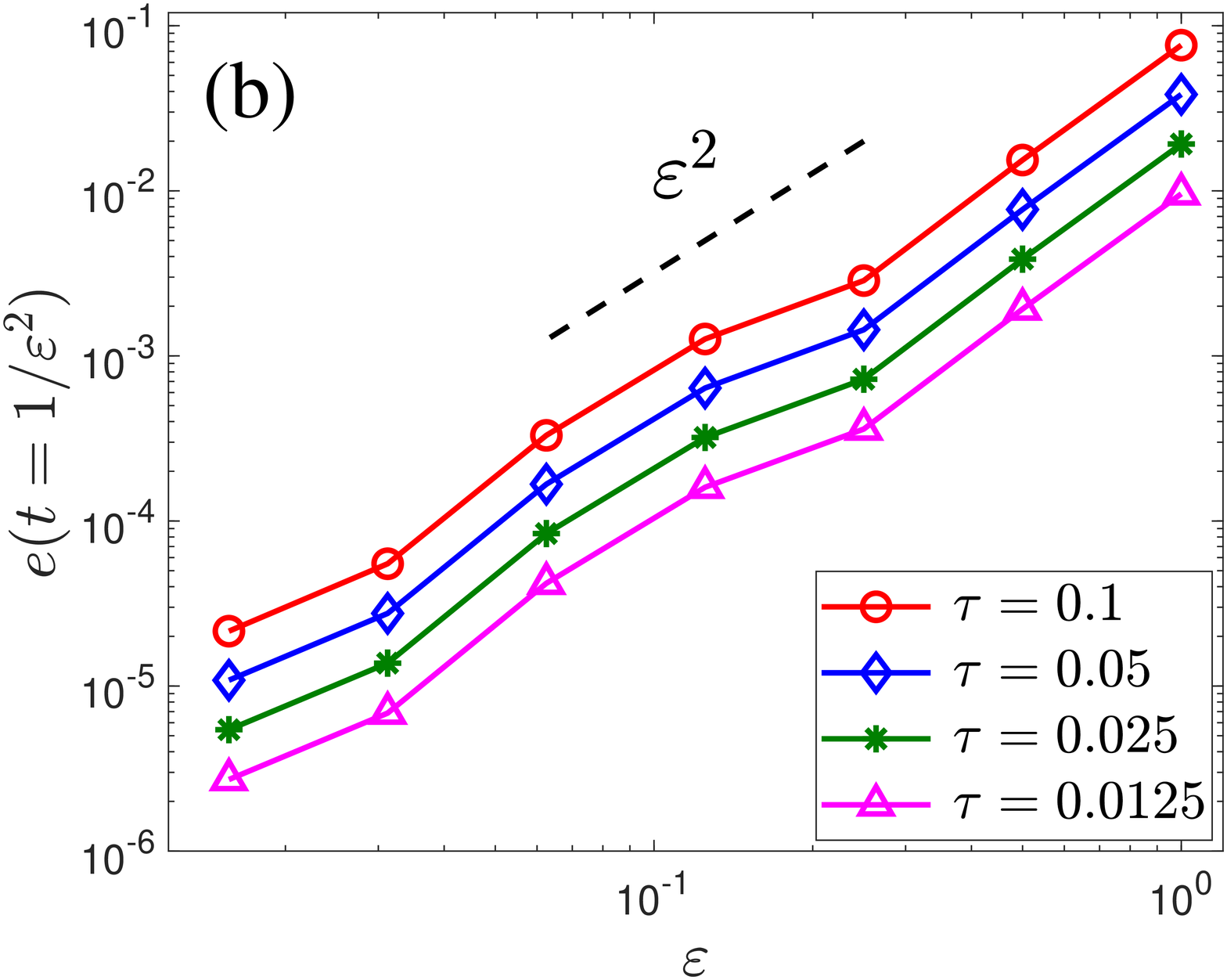}}
\end{minipage}
\caption{Long-time temporal errors of the LEI-FP method \eqref{eq:psifull1}--\eqref{eq:psifull2} with \eqref{eq:wfull} for the SGE \eqref{eq:WNE} in 1D at $t = 1/\eps^2$: (a) convergence rates in $\tau$, and (b) convergences rate in $\eps$.}
\label{fig:1D_tau}
\end{figure}

From Fig. \ref{fig:1D_h} and Fig. \ref{fig:1D_tau} and additional numerical results not shown here for brevity, we have the following observations:

(i) The LEI-FP method converges uniformly for $0 < \eps \leq 1$ in space with exponential convergence rate (cf. Fig. \ref{fig:1D_h}).

(ii) For any fixed $\eps = \eps_0 >0$, the LEI-FP method \eqref{eq:psifull1}--\eqref{eq:psifull2} with \eqref{eq:wfull} is first-order in time (cf. each line in Fig. \ref{fig:1D_tau}(a)) and the temporal errors behave like $O(\eps^2)$ for the fixed time step $\tau$ (cf. each line in Fig. \ref{fig:1D_tau}(b)).

(iii) The numerical result confirms the improved uniform error bound \eqref{eq:error_full} for the full-discretization.

\subsection{Long-time dynamics in 2D}
In this subsection, we show an example in 2D with the irrational aspect ratio of the domain $ (x, y) \in\Omega= (0, 1) \times (0, 2\pi)$. The initial data is chosen as
\begin{equation*}
\phi(x, y) = \frac{2}{2 + \cos^2(2\pi x+y)},\quad \gamma(x) = \frac{2}{2 + 2\cos^2(2\pi x+y)}.
\end{equation*}

\begin{figure}[ht!]
\begin{minipage}{0.49\textwidth}
\centerline{\includegraphics[width=6.5cm,height=5.5cm]{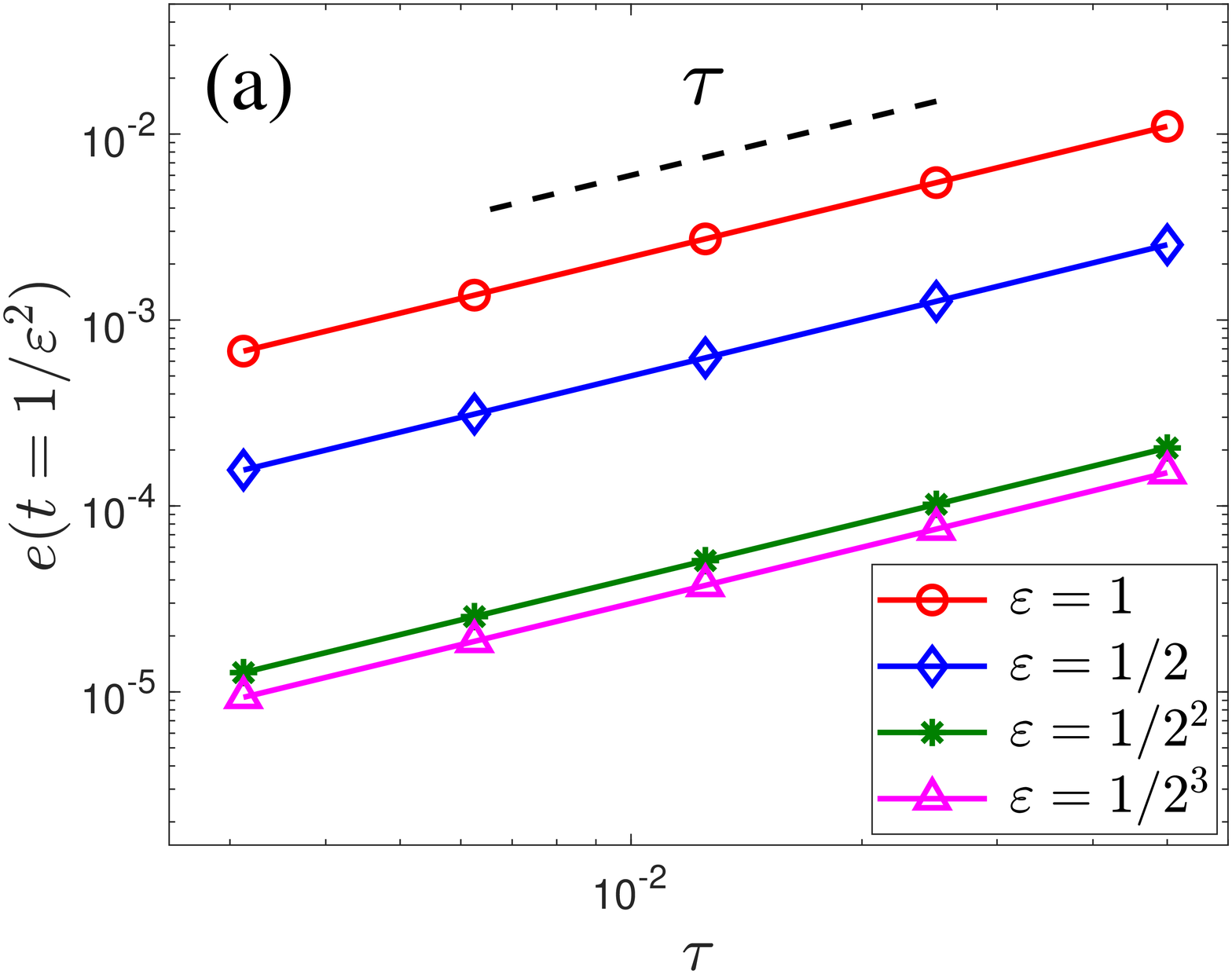}}
\end{minipage}
\begin{minipage}{0.49\textwidth}
\centerline{\includegraphics[width=6.5cm,height=5.5cm]{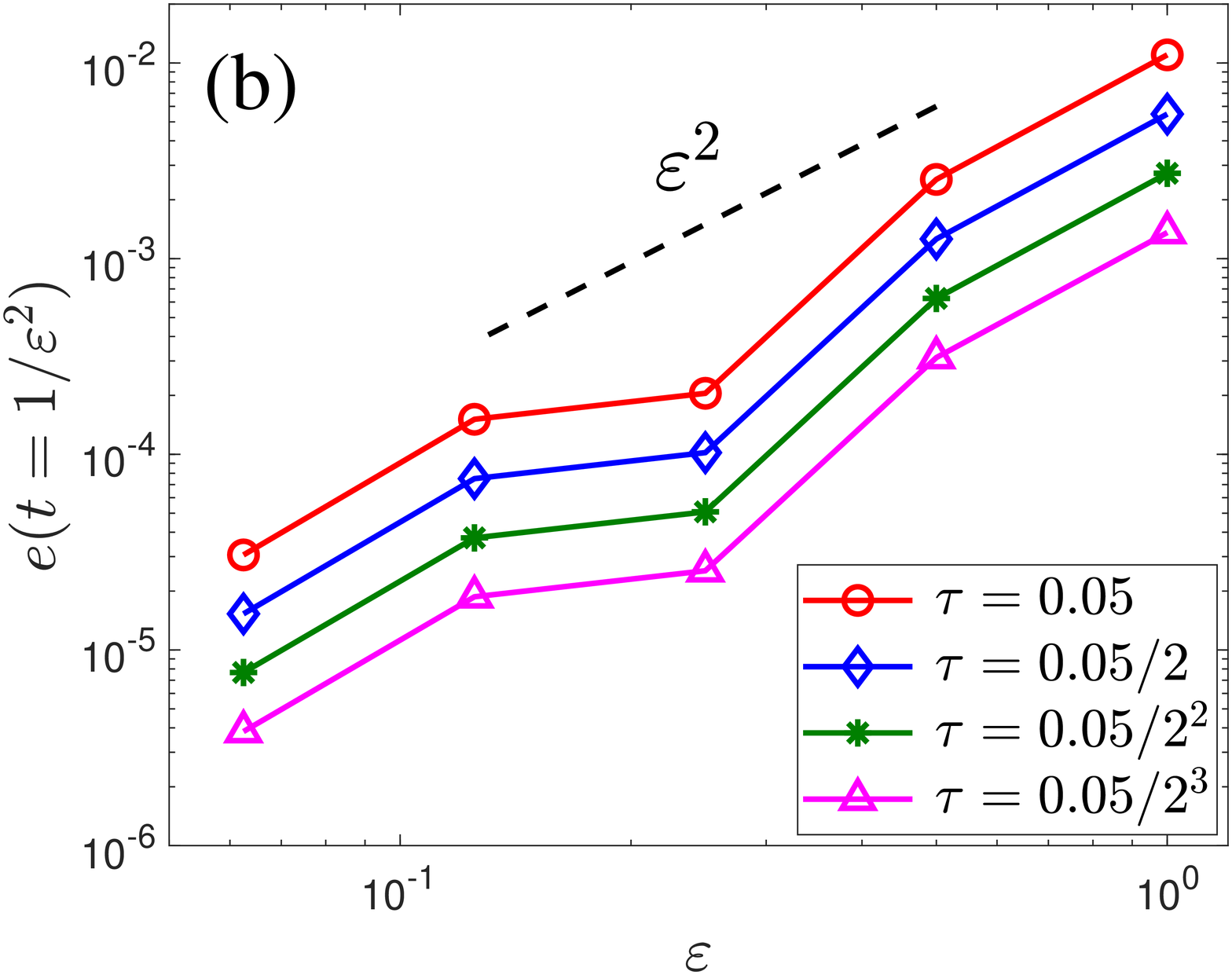}}
\end{minipage}
\caption{Long-time temporal errors of the LEI-FP method for the SGE \eqref{eq:WNE} in 2D at $t = 1/\eps^2$: (a) convergence rates in $\tau$, and (b) convergences rate in $\eps$.}
\label{fig:2D_temporal}
\end{figure}

 Fig. \ref{fig:2D_temporal} presents the long-time temporal errors of the LEI-FP method for the SGE \eqref{eq:WNE} in 2D up to the time at $O(1/\eps^2)$ in the domain with irrational aspect ratio, which indicates that the LEI-FP method is first-order in time and the improved uniform error bound behaves like $O(\eps^2\tau)$ up to the time at $O(1/\eps^2)$. The numerical result indicates that the improved uniform error bound is also suitable for the higher dimensional domain with irrational aspect ratio.

\subsection{Dynamics of the oscillatory sine--Gordon equation}
In this subsection, we present the numerical result for the oscillatory SGE \eqref{eq:HOE} in 1D to confirm the improved error bound \eqref{eq:HOE_error}. We choose the initial data as
\begin{equation*}
\phi(x) = x^2(x-1)^2+3,\quad \gamma(x) = x(x-1)(2x-1), \quad x \in\Omega= (0, 1). 	
\end{equation*}
Here, the regularity of the initial data is enough to ensure the improved error bound \eqref{eq:HOE_error}. The `exact' solution is obtained numerically by the LEI-FP method with a very fine mesh size $h_e = 1/128$ and a very small time step size $\kappa_e = 10^{-6}$. The spatial mesh size is chosen sufficiently small and we test the temporal errors of the LEI-FP method for the oscillatory SGE \eqref{eq:HOE}. The temporal errors are displayed at $t = 1$ with different $\eps$. 

\begin{table}
\caption{Temporal errors of the LEI-FP method for the oscillatory SGE \eqref{eq:HOE} in 1D.}
\centering
\renewcommand\arraystretch{1.3}
\begin{tabular}{ccccccc}
\hline
$e_1(t = 1)$ &$\tau_0 = 0.1 $ & $\tau_0/4 $ &$\tau_0/4^2 $ & $\tau_0/4^3 $ & $\tau_0/4^4$ & $\tau_0/4^5$  \\
\hline
$\varepsilon_0 = 1$ & \bf{1.82E-1} & 4.72E-2 & 1.19E-2 & 2.99E-3 & 7.45E-4 & 1.85E-4 \\
order & \bf{-} & 0.97 & 0.99 & 1.00 & 1.00 & 1.01 \\
\hline
$\varepsilon_0 / 2 $  & 6.97E-2 & \bf{1.79E-2} & 4.51E-3 & 1.13E-3 & 2.82E-4 & 6.99E-5 \\
order & -  & \bf{0.98} & 0.99 & 1.00 & 1.00 & 1.01 \\
\hline
$\varepsilon_0 / 2^2 $ & 7.48E-1 & 1.82E-2 & \bf{4.29E-3} & 1.06E-3 & 2.64E-4 & 6.55E-5 \\
order & -  & 2.68 & \bf{1.04} & 1.01 & 1.00 & 1.01 \\
\hline
$\varepsilon_0 / 2^3 $ & 4.80 & 3.32E-1 & 8.52E-3 & \bf{2.21E-3} & 5.59E-4 & 1.39E-4 \\
order & -  & 1.93 & 2.64  & \bf{0.97} & 0.99 & 1.00 \\
\hline
$\varepsilon_0 / 2^4$ & 6.53E-1 & 5.10E-1 & 6.41E-2 & 3.17E-3 & \bf{8.11E-4} & 2.03E-4 \\
order & - & 0.18 & 1.50 & 2.17 & \bf{0.98} & 1.00 \\
\hline
$\varepsilon_0 / 2^5$ & 2.32E-1 & 5.54E-2 & 5.23E-2 & 1.62E-2 & 7.01E-4 & {\bf 1.69E-4} \\
order & - & 1.03 & 0.04 & 0.85 & 2.26 &  \bf{1.03} \\
\hline
\end{tabular}
\label{tab:HOE_tau}
\end{table}

Table \ref{tab:HOE_tau} lists the temporal errors of the LEI-FP method for the oscillatory SGE \eqref{eq:HOE} in 1D, which indicates that the first-order convergence can only be observed when $\kappa\lesssim \eps^2$ (cf. the upper triangle above the diagonal with bold letters) and the temporal errors are independent of $\eps$ under this condition. The numerical results confirm the improved error bound \eqref{eq:HOE_error} and to demonstrate that they are sharp.

\section{Conclusions}
The Lawson-type exponential integrator Fourier pseudospectral (LEI-FP) method was applied to numerically solve the sine--Gordon equation with small initial data or weak nonlinearity. By separating a linear part from the sine function and employing the regularity compensation oscillation (RCO) technique, the improved uniform error bounds for the semi-discretization and full-discretization were carried out at $O(\eps^2\tau)$ and $O(h^m+\eps^2\tau)$, respectively, for the long-time dynamics of the sine--Gordon equation up to the time at $T/\eps^2$ with $T > 0$ fixed. The improved uniform error bounds for the long-time problem are extended to the oscillatory sine--Gordon equation up to the fixed time $T$. Numerical results were presented to confirm the improved error bounds and to demonstrate that they are sharp.

\section*{Acknowledgements}
The authors gratefully acknowledge funding from the European Research Council (ERC) under the European Union's Horizon 2020 research and innovation programme (grant agreement No. 850941).

\bibliographystyle{plain}

\end{document}